\documentclass[11pt]{article}
\usepackage[T1]{fontenc}
\usepackage{lmodern}

\usepackage{amsmath,amsthm,amssymb}
\usepackage[usenames,dvipsnames]{xcolor}
\usepackage{enumerate}
\usepackage{graphicx}
\usepackage{cite}
\usepackage{comment}
\usepackage{oands}
\usepackage{tikz}
\usepackage{changepage}
\usepackage{bbm}
\usepackage{mathtools}
\usepackage[margin=1in]{geometry}

\usepackage[pdftitle={Inverting the operation of conditioning a branching process on extinction},
  pdfauthor={Ewain Gwynne and Jiaqi Liu},
colorlinks=true,linkcolor=NavyBlue,urlcolor=RoyalBlue,citecolor=PineGreen,bookmarks=true,bookmarksopen=true,bookmarksopenlevel=2,unicode=true,linktocpage]{hyperref}

\theoremstyle{plain}
\newtheorem{thm}{Theorem}[section]

\newtheorem{lem}[thm]{Lemma}
\newtheorem{prop}[thm]{Proposition}

\def\@rst #1 #2other{#1}
\newcommand\MR[1]{\relax\ifhmode\unskip\spacefactor3000 \space\fi
  \MRhref{\expandafter\@rst #1 other}{#1}}
\newcommand{\MRhref}[2]{\href{http://www.ams.org/mathscinet-getitem?mr=#1}{MR#2}}

\theoremstyle{definition}
\newtheorem{defn}[thm]{Definition}
\newtheorem{remark}[thm]{Remark}

\newtheorem{example}[thm]{Example}

\numberwithin{equation}{section}

\def\alb#1\ale{\begin{align*}#1\end{align*}}
\def\allb#1\alle{\begin{align}#1\end{align}}
\newcommand{\eqb}{\begin{equation}}
\newcommand{\eqe}{\end{equation}}
\newcommand{\eqbn}{\begin{equation*}}
\newcommand{\eqen}{\end{equation*}}

\newcommand{\BB}{\mathbbm}
\newcommand{\ol}{\overline}
\newcommand{\bd}{\mathbf}
\newcommand{\ep}{\varepsilon}
\newcommand{\wt}{\widetilde}
\newcommand{\wh}{\widehat}
\newcommand{\mcl}{\mathcal}
\newcommand{\bdy}{\partial}

\newcommand{\cc}{{\mathbf{c}}}

\let\originalleft\left
\let\originalright\right
\renewcommand{\left}{\mathopen{}\mathclose\bgroup\originalleft}
\renewcommand{\right}{\aftergroup\egroup\originalright}

\title{Inverting the operation of conditioning a branching process on extinction}
 \date{ }
 \author{
\begin{tabular}{c} Ewain Gwynne\\[-3pt]\small University of Chicago \end{tabular}
\begin{tabular}{c} Jiaqi Liu\\[-3pt]\small Lehigh University \end{tabular}  
}

\begin{document}

\maketitle

\begin{abstract}
It is well-known that conditioning a supercritical (multi-type) branching process on the event that it eventually becomes extinct yields a subcritical branching process. We study the corresponding inverse problem: given a subcritical branching process, does there exist a supercritical branching process with the property that when we condition it on extinction, we get back the original subcritical branching process? We show that such a supercritical branching process (which we call a conjugate branching process) exists under mild hypotheses on the original subcritical branching process. We also show by example that if there are at least two types, then the conjugate branching process is not necessarily unique. Our results are relevant to the problem of constructing natural random planar maps whose scaling limit is given by supercritical Liouville quantum gravity. Moreover, conjugate branching processes can also be used to give alternative evolutionary hypotheses in cancer modeling. 
\end{abstract}

\tableofcontents

\bigskip
\noindent\textbf{Acknowlegments.} {We thank two anonymous referees for helpful comments on an earlier version of this paper.} We thank Morris Ang, Manan Bhatia, Xinxin Chen, Hui He, Alex McAvoy, Robin Pemantle, Jason Schweinsberg, Xin Sun, and Jinwoo Sung for helpful discussions. E.G.\ was partially supported by NSF grant DMS-2245832. 

\section{Introduction}
\label{sec-intro}

\subsection{Overview}
\label{sec-overview}

Let $d\in\BB N$ and consider a $d$-type branching process $X = \{X(n) \}_{n\in\BB N_0}$, where $X(n) = (X_1(n) , \dots, X_d(n))$ is the vector which gives the number of individuals of types $k\in \{1,\dots,d\}$ in the $n$th generation. 
For $k\in \{1,\dots,d\}$, let $\BB P_k$ be the probability measure where we start with a single individual of type $k$ (i.e., $X_k(0) = 1$ and $X_j(0) = 0$ for $j\not=k$), and let $\BB E_k$ be the corresponding expectation. 

A classical fact in the theory of branching processes says that conditioning a supercritical branching process on extinction gives a subcritical branching process. More precisely, let $E$ be the event that $X(n) = 0$ for all sufficiently large $n$, and assume that $0 < \BB P_k[E] < 1$ for each $k\in\{1,\dots,d\}$. Then the conditional law of $X$ given $E$ is that of a subcritical branching process. See Proposition~\ref{prop-cond-on-ext} below for a precise statement and~\cite[Theorem 1]{jl-cond-on-ext} for a more general result. The branching process obtained by conditioning on extinction is sometimes called the \textbf{conjugate branching process}.

In this paper, we will investigate the extent to which the operation of conditioning a branching process on extinction is invertible. In other words, we will begin with a subcritical branching process $X$, and ask whether there exists a supercritical branching process $\wh X$ whose conjugate is $X$. As we explain in Section~\ref{sec-lqg}, this question is relevant to the problem of constructing random planar maps whose scaling limit is described by supercritical (a.k.a.\ strongly coupled) Liouville quantum gravity. As an application in biology, we also explain in Section~\ref{sec-bio} how conjugate branching processes can be useful in giving alternative evolutionary hypotheses in cancer modeling. However, we emphasize that the proofs in this paper are completely elementary: no knowledge of Liouville quantum gravity or biology is needed to understand the paper. 
 
\begin{defn} \label{def-conjugate}
Let $X = \{X(n) : n\in\BB N\}$ be a subcritical $d$-type branching process such that the $\BB P_k$-extinction probability is one for each $k\in \{1,\dots,d\}$. 
A $d$-type branching process $\wh X$ is called a \textbf{conjugate branching process} for $X$ if $\wh X$ does not have the same law as $X$ and for each $k\in \{1,\dots,d\}$, the conditional law of $\wh X$ (started from a single individual of type $k$) given that it eventually becomes extinct is the same as the $\BB P_k$-law of $X$. 
\end{defn}
Our main result (Theorem~\ref{thm-fixed-gf}) states that under mild conditions, a subcritical multi-type branching process admits a conjugate branching process. By basic branching process theory, this is equivalent to the existence of a fixed point for the generating function which lies in $(1,\infty)^d$. We note that the proof of the single-type case of our result is straightforward (see Proposition~\ref{prop-single-type}), but non-trivial topological arguments, based on the Brouwer fixed point theorem, are needed to treat the case when $d\geq 2$. 
 
We also give an example to show that the conjugate branching process is not unique, even for a ``generic'' subcritical branching process (Example~\ref{counterexample}).
Finally, we extend our result to the case of branching processes with countably many types, see Section~\ref{sec-infty}. 

\begin{remark}
There are various ways to make sense of conditioning a subcritical or critical branching processes on the (zero-probability) event that it never becomes extinct. The first approach was introduced by Kolmogorov \cite{kolmogorov1938losung} and Yaglom \cite{yaglom1947certain}, who considered the limiting behavior of the process conditioned on non-extinction up to generation $n$ as $n \to \infty$. In the single type subcritical case, Yaglom \cite{yaglom1947certain} showed that as $n \to \infty$, the distribution of $\{ X(n) | X(n)>0 \}$ converges to a proper distribution. In the single type critical case, Yaglom \cite{yaglom1947certain} showed that as $n \to \infty$, the distribution of the rescaled process $\{X_n/n|X_n>0\}$ converges to an exponential distribution. The multi-type analog was studied by \cite{harris1963theory, jirina1957asymptotic, joffe1967, mullikin1963} in the case of finitely many types, and \cite{seneta1966quasi} in the countable-type case. The second approach is conditioning the process on not being extinct in the distant future, i.e. $\{X(n) | X(n+k)>0\}$ as $k \to \infty$. This leads to the so-called $Q$-process, and was first studied by \cite{lamperti1968conditioned} in the single type case, and later by \cite{dallaporta2008q} in the multi-type case. For other ways to condition a (sub)critical branching process on non-extinction, see Section 2.1 of \cite{Penisson-thesis}. The notion of a conjugate branching process considered in the present paper is distinct from these concepts: a conjugate of a subcritical branching process is a supercritical branching process, whereas conditioning a subcritical process on non-extinction usually leads to the stationary behavior of the process and the resulting object is typically not a branching process.
\end{remark}

\subsection{Main result}
\label{sec-result}

Fix $d\in\BB N$ and consider a $d$-type branching process $X = \{X(n) \}_{n\in\BB N_0}$. 
We define the \textbf{generating function} $f : (0,\infty]^d \to (0,\infty]^d$ by
\eqb  \label{eqn-generating-function}
f_k(q) := \BB E_k\left[ \prod_{j=1}^d q_j^{X_j(1)} \right] \quad \text{and} \quad f(q) = (f_1(q) ,\dots,f_d(q)) , \quad \text{where} \quad q = (q_1,\dots,q_d) .
\eqe
Note that we allow $q_j = \infty$ for some values of $j$, in which case we interpret $\infty^0 = 1$ and $\infty^x = \infty$ for $x > 0$.   
For $n\in\BB N$, let $f^{(n)}$ be the composition of $f$ with itself $n$ times.
Equivalently, for $k\in \{1,\dots,d\}$, then $k$th component of $f^{(n)}$ is given by 
\eqb \label{eqn-gen-func-n}
f_k^{(n)}(q) := \BB E_k\left[ \prod_{j=1}^k q_j^{X_j(n)} \right]  . 
\eqe 
 
We also define the $d\times d$ matrix $M$ whose $k,j$ entry is 
\eqb \label{eqn-mean-matrix} 
 M_{k,j} :=   \BB E_k[X_j(1)]  .
\eqe 
We assume that each $M_{k,j}$ is finite. 
Then for $n\in\BB N$, the $k,j$ entry of $M^n$ is $\BB E_k[X_j(n)]$. Recall that the branching process is \textbf{subcritical} (resp.\ \textbf{critical}, \textbf{supercritical}) if the largest eigenvalue $\rho$ of $M$ is less than one (resp.\ equal to one, greater than one). 

The following are standard non-degeneracy conditions for multi-type branching processes, see, e.g.,~\cite[Section 6]{an-branching-processes}. 

\begin{defn} \label{def-bp-basic}  
We say that $X$ is \textbf{positive regular} if there exists $N\in\BB N$ such that 
\eqbn
\BB P_k\left[ X_j(N) \geq 1 \right] > 0 ,\quad\forall k,j \in \{1,\dots,d\} ,
\eqen
which is equivalent to the condition that the matrix $M^N$ has all positive entries.
We say that the branching process is \textbf{non-singular} if there exists $k \in\{1,\dots,d\}$ such that
\eqbn
\BB P_k\left[ \sum_{j=1}^d X_j(1) \geq 2 \right] >  0 .
\eqen 
\end{defn}

It is easy to see that if $X$ is a subcritical, positive regular, and non-singular multi-type branching process, then the existence of a conjugate branching process for $X$ in the sense of Definition~\ref{def-conjugate} is equivalent to the existence of a point $q \in (1,\infty)^d$ for which $f(q) = q$ (see Proposition~\ref{prop-conjugate} below). 
The main result of this paper is the following theorem. 
 
\begin{thm} \label{thm-fixed-gf}
Consider a $d$-type branching process $X = \{X(n)\}_{n\in\BB N_0}$ with generating function $f$ satisfying the following hypotheses. 
\begin{enumerate}[I.]
\item $X$ is subcritical, i.e., $\rho <1$. 
 \label{item-gf-extinction}
\item $X$ is positive regular and non-singular (Definition~\ref{def-bp-basic}).  \label{item-gf-pos} 
\item For each $k\in\BB N$, the $k$th component $f_k$ of the generating function is a continuous function from $(0,\infty]^d$ to $(0,\infty]$ (with $(0,\infty]$ equipped with the topology of a half-open interval). \label{item-gf-cont}
\end{enumerate}
Then there exists a conjugate branching process for $X$, equivalently, there exists $q \in (1,\infty)^d  $ such that $f(q) = q$. 
\end{thm} 

Hypothesis~\ref{item-gf-extinction} is necessary for Theorem~\ref{thm-fixed-gf} to hold since a supercritical branching process conditioned on extinction is always a subcritical branching process. Hypothesis~\ref{item-gf-pos} just gives standard regularity conditions. The most restrictive hypothesis is Hypothesis~\ref{item-gf-cont}. To see why this hypothesis cannot be removed, we note that it is possible that the first two hypotheses hold but that $f(q) = (\infty,\dots,\infty)  \not= q$ for all $q \in (1,\infty)^d  $ (in which case no conjugate branching process exists). For a concrete example, consider the single-type branching process with $\BB P[X(1) = m] = a m^{-3}$ for $m\geq 1$, and $\BB P[X(1) = 0] = b$ for appropriate constants $a,b>0$. We note that Hypothesis~\ref{item-gf-cont} is automatically satisfied if $X$ is positive regular and non-singular and $f_k(q) < \infty$ for all $k\in \{1,\dots,d\}$ and all $q\in (0,\infty)^d$.

\begin{example}[Non-uniqueness] \label{counterexample}
The conjugate branching process in Theorem~\ref{thm-fixed-gf} appears to be unique for many simple examples of subcritical branching processes. However, it is not unique in general, even for a ``generic'' choice of offspring distribution. As a fairly minimal example, consider the case when $d=2$ and the offspring distribution is given by
\alb
&\BB P_k\left[ X_k(1) =1 , X_{3-k}(1) = 0 \right] = \BB P_k\left[ X_k(1) =2 , X_{3-k}(1) = 0 \right]  =  \frac14 , \\
&\BB P_k\left[ X_k(1) = X_{3-k}(1) = 2 \right] = \frac{1}{100}, \quad
\BB P_k\left[ X_k(1) = X_{3-k}(1) = 0\right] = \frac{49}{100} ,\quad\forall k \in \{1,2\} .
\ale
The branching process clearly satisfies the hypotheses of Theorem~\ref{thm-fixed-gf}, and one can check that the equation $f(q_1,q_2) = (q_1,q_2)$ has three solutions in $(1,\infty)^2$. Furthermore, this example is robust in the sense that if we slightly perturb the offspring distribution (e.g., we could add a small positive chance to have more than three offspring of each type) then there are still three fixed points of the generating function in $(1,\infty)^2$. See Figure~\ref{fig-counterexample} for a visualization of why this example works. We also have a similar example where $d=3$ and there are seven fixed points in $(1,\infty)^d$.  
In general, we expect that for a fixed $d$, the number of fixed points of $f$ in $(1,\infty)^d$ can be arbitrarily large.
\end{example} 

\begin{figure}[ht!]
\begin{center}
\includegraphics[width=0.3\textwidth]{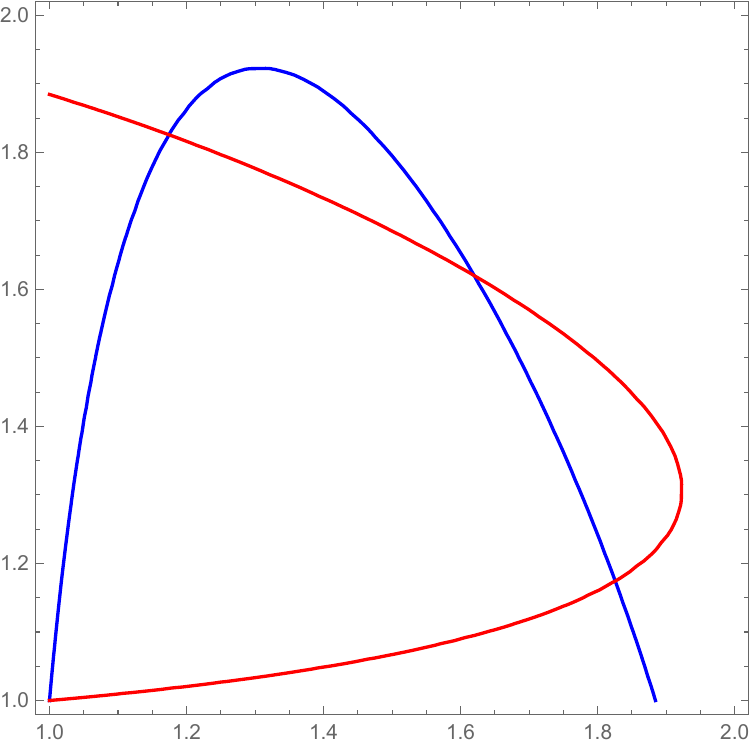}  
\caption{\label{fig-counterexample} Plot of the solution sets of $f_1(q_1,q_2) = q_1$ (blue) and $f_2(q_1,q_2) = q_2$ (red) in the setting of Example~\ref{counterexample}, made using Mathematica. The intersection points of these sets are the fixed points of $f$. 
}
\end{center}
\end{figure}

Finally, we note that the analog of Theorem~\ref{thm-fixed-gf} in the case of a single-type branching process is easy to prove, and one also gets that the fixed point is unique. 

\begin{prop} \label{prop-single-type}
Let $X= \{X(n)\}_{n \geq 0}$ be a single-type branching process started from a single individual and let $f$ be its generating function. Assume that 
\begin{enumerate}[I.]
\item $X$ is subcritical, i.e., $\BB E[X(1)] < 1$.  \label{item-single-type-sub}
\item We have $\BB P[X(1) \geq 2] > 0$.  \label{item-single-type-convex} 
\item $f$ is continuous from $(0,\infty]$ to $(0,\infty]$ (with $(0,\infty]$ equipped with the topology of a half-open interval). \label{item-single-type-cont}
\end{enumerate}
Then there exists a unique conjugate branching process for $X$. Equivalently, there exists a unique value of $ q > 1$ for which $f(  q) =  q$. 
\end{prop}
\begin{proof}
Since $f(1) = 1$ and by Hypothesis~\ref{item-single-type-cont}, there exists $\ep > 0$ such that $f(q) < \infty$ for all $q \in [1,1+\ep]$. 
We have $\BB E[X(1)] = f'(1) < 1$, so by Taylor expansion it follows that if $s > 1$ is sufficiently close to 1, then $f(s) < s$. By Hypothesis~\ref{item-single-type-convex}, there exists a sufficiently large $t > 1$ for which $f(t) > t$. By Hypothesis~\ref{item-single-type-cont}, there exists $  \ol q \in [s,t]$ for which $f( \ol q) =   \ol q$. Since $f''(q) > 0$ for all $q > 0$ for which $f(q) < \infty$, we have that $f$ is convex on the maximal interval where it is finite-valued. Hence there are at most two values of $q$ for which $f(q) = q$, namely $1$ and $ \ol q$. 
\end{proof}

\subsection{Proof outline}
\label{sec-outline}

Our proof of Theorem~\ref{thm-fixed-gf}, given in Section~\ref{sec-proof}, is via an elementary topological argument. 
The basic strategy is to reduce to a situation where we can apply the Brouwer fixed point theorem. To this end, we will consider the basins of attractions defined by
\eqbn
S_0 := \left\{ \lambda \in [0,\infty]^d : \lim_{n\to\infty} \log f^{(n)}(e^\lambda) = 0 \right\} \;\; \text{and} \;\;
S_\infty := \left\{ \lambda \in [0,\infty]^d : \limsup_{n\to\infty} |\log f^{(n)}(e^\lambda)| = \infty \right\} ,
\eqen
where here logarithms and exponentials are applied componentwise, e.g., $e^\lambda = (e^{\lambda_1} ,\dots,e^{\lambda_d})$. 

With some work, we can show that each of $S_0$ and $S_\infty$ is relatively open in $[0,\infty]^d$ (Lemma~\ref{lem-basin-open}). This is the part of the proof which requires the continuity assumption on $f$ appearing in Theorem~\ref{thm-fixed-gf}. We then consider the compact set $\Lambda := [0,\infty)^d\setminus (S_0\cup S_\infty)$.
Using the openness of $S_0$ and $S_\infty$ and the monotonicity properties of $f$ established in Section~\ref{sec-mono}, we show that each infinite ray started from 0 which is contained in $[0,\infty)^d$ intersects $\Lambda$ exactly once (Lemma~\ref{lem-ray-intersect}). We note that for this to be true, it is important that $S_0$ and $S_\infty$ are defined using $f(e^\lambda)$ rather than $f(q)$ (see Lemma~\ref{lem-unstable-scaled}). 

 From this, we deduce that $\Lambda$ is a deformation retract of a collar neighborhood $A$ of $\Lambda$ which has the topology of a closed disk (Lemma~\ref{lem-retract}). The Brouwer fixed point theorem implies that every continuous function from $A$ to itself has a fixed point, so the same is true for $\Lambda$. It is clear that the function $g(\lambda) = \log f(e^\lambda)$ (with $\log$ applied componentwise) maps $\Lambda$ to itself. From this we get that $g$ has a fixed point in $\Lambda$, which implies that $f$ has a fixed point in $\exp(\Lambda)\subset [1,\infty)^d \setminus \{(1,\dots,1)\}$. The positive regularity of the branching process implies that the fixed point is in fact in $(1,\infty)^d$ (Lemma~\ref{lem-fixed-pos}).
 
In Section~\ref{sec-infty}, we establish a version of our theorem (Theorem~\ref{thm-fixed-pt-infty}) for a countable-type branching process by taking a limit as the number of types goes to $\infty$. In Section~\ref{sec-lqg}, we discuss the applications of our results to random planar maps and Liouville quantum gravity. In Section~\ref{sec-bio}, we discuss how our result can be used to give alternative evolutionary hypotheses in cancer modeling.

\section{Proof of main theorem}
\label{sec-proof}

\subsection{Preliminaries on branching processes}
\label{sec-prelim}
We review some fairly standard facts about branching processes which are needed for our proofs. 
The following is a precise statement of the fact that conditioning a (multi-type) supercritical branching process on extinction gives a subcritical branching process.

\begin{prop} \label{prop-cond-on-ext}
Let $X = \{X(n) \}_{n\in\BB N_0}$ be a positive regular, non-singular $d$-type branching process  (Definition~\ref{def-bp-basic}) with generating function $f$. Let $E$ be the event that the process eventually becomes extinct. For $k\in \{1,\dots,d\}$, define the extinction probability 
\eqbn
a_k := \BB P_k\left[ E \right] .
\eqen 
Then the only solutions of $f(q) =q$ in the unit cube $[0,1]^d$ are $a = (a_1,\dots,a_d)$ and $(1,\dots,1)$. 

Suppose that $0 < a_k < 1$ for every $k \in \{1,\dots,d\}$. Then for each $k\in\{1,\dots,d\}$, the $\BB P_k$-conditional law of $X$ given $E$ is that of a branching process with offspring distribution determined by 
\eqb \label{eqn-cond-on-ext}
\BB E_k\left[ F(X_1(1) , \dots, X_d(1)) \,|\, E \right] 
 = \frac{1}{a_k} \BB E_k\left[ F(X_1(1),\dots,X_d(1)) \prod_{j=1}^d a_j^{X_j(1)} \right]   
\eqe 
for every bounded measurable function $F : \BB N_0^d \to \BB R$. Moreover, this branching process is subcritical. 
\end{prop}
\begin{proof}
The proof is elementary: for details, see~\cite[Theorem 2, Section V.3]{an-branching-processes} for the properties of $a$ and, e.g.,~\cite[Theorem 1]{jl-cond-on-ext} for a description of the conditional law of $X$ given $E$ in a somewhat more general setting. 
\end{proof}

We will be interested in the inverse operation of conditioning on extinction. 

\begin{prop} \label{prop-conjugate}
Let $X = \{X(n) \}_{n\in\BB N_0}$ be a positive regular, non-singular $d$-type branching process (Definition~\ref{def-bp-basic}) with generating function $f$. Assume that the extinction probability started from any state is one. 
Let $A \in [1,\infty)^d$ solve $f(A) = A$. Define a new branching process $\wh X = \{\wh X(n)  \}_{n\in\BB N_0}$ with offspring distribution determined by
\eqb \label{eqn-conjugate-def}
\wh{\BB E}_k\left[ F(X_1(1) , \dots, X_d(1))   \right] 
 = \frac{1}{A_k} \BB E_k\left[ F(X_1(1),\dots,X_d(1)) \prod_{j=1}^d A_j^{X_j(1)} \right] ,\quad\forall k\in \{1,\dots,d\}    
\eqe  
for every bounded measurable function $F : \BB N_0^d \to \BB R$. Then the extinction probability for $\wh X$ started from an individual of type $k$ is $1/A_k$, and the $\wh{\BB P}_k$-conditional law of $\wh X$ given the event that it eventually becomes extinct is the same as the $\BB P_k$-law of $X$.
\end{prop}
\begin{proof}
Since $f(A) =A$, we see that~\eqref{eqn-conjugate-def} defines the offspring distribution for a branching process $\wh X$. Since the law of the first $n$ generations of $\wh X$ is mutually absolutely continuous with respect to the law of the first $n$ generations of $X$, we obtain that $\wh X$ is also positive regular and non-singular. 
Let $\wh f  $ be the generating function for $\wh X$ and write $1/A = (1/A_1,\dots,1/A_k) \in [0,1]^d$. By~\eqref{eqn-conjugate-def}, $\wh f(1/A) = 1/A$. By Proposition~\ref{prop-cond-on-ext}, it follows that $1/A$ is the vector of extinction probabilities for $\wh X$ and that the conditional law of $\wh X$ given extinction is the same as the law of $X$. 
\end{proof}

Our next lemma implies in particular that in the setting of Proposition~\ref{prop-conjugate}, the fixed point $A$ must either be equal to $(1,\dots,1)$ or be in $(1,\infty)^d$.

\begin{lem} \label{lem-fixed-pos}
Consider a $d$-type branching process $X = \{X(n)\}_{n\in\BB N_0}$ for which the extinction probability started from any state is one. Let $f$ be the corresponding generating function and let $q \in (0,\infty)^d$ solve $f(q) = q$. Then   
\eqb \label{eqn-fixed-pos-geq}
q_k \geq 1, \quad \forall k \in \{1, \dots, d\}  .
\eqe 
Moreover, if the branching process is positive regular, then either $q_k = 1$ for every $k\in \{1,\dots,d\}$ or $q_k > 1$ for every $k \in \{1,\dots,d\}$.
\end{lem}
\begin{proof} 
Let $q \in (0,\infty)^d$ for which $f(q) = q$. 
By our extinction hypothesis, for every $k \in   \{1,\dots,d\}$, it holds $\BB P_k$-a.s.\ that  $\lim_{n\to\infty} \prod_{j=1}^d q_j^{X_j(n)} = 1$. 
We have $f^{(n)}(q) = q$ for every $n\in\BB N$. 
Therefore, we can use Fatou's lemma to get that for every $k \in \{1,\dots,d\}$, 
\alb
q_k 
= \lim_{n\to\infty} f_k^{(n)}(q) 
= \lim_{n\to\infty} \BB E_k\left[ \prod_{j=1}^d  q_j^{X_j(n)} \right] 
\geq \BB E_k\left[ \lim_{n\to\infty} \prod_{j=1}^d  q_j^{X_j(n)} \right] 
= 1 .
\ale
Thus,~\eqref{eqn-fixed-pos-geq} holds. 

Now assume that the branching process is positive regular and there exists $j_* \in \{1,\dots,d\}$ such that $q_{j_*} > 1$.  
Let $N\in\BB N$ be such that $\BB P_k[X_{j_*}(N) \geq 1 ] > 0$ for every $k\in \{1,\dots,d\}$. 
Then for every $k\in \{1,\dots,d\}$, 
\alb
q_k 
= f_k^{(N)}(q) 
= \BB E_k\left[ \prod_{j=1}^d  q_j^{X_j(N)} \right] 
\geq q_{j_*} \BB P_k\left[ X_{j_*}(N) \geq 1 \right] + \left( 1 - \BB P_k\left[ X_{j_*}(N) \geq 1 \right] \right)  
 > 1 
\ale
since $q_{j_*} > 1$ and $ \BB P_k\left[ X_{j_*}(N) \geq 1 \right] > 0$.
\end{proof}

The following elementary consequence of positive regularity and non-singularity will be used frequently in our proofs.

\begin{lem} \label{lem-gf-pos}
Let $X = \{X(n)\}_{n\in\BB N_0}$ be a $d$-type branching process which is positive regular and non-singular (Definition~\ref{def-bp-basic}). There exists $N\in\BB N$ such that for each $k,j  \in \{1,\dots,d\}$, 
\eqb \label{eqn-gf-pos} 
p_{k,j}(N) > 0 \quad \text{where} \quad
p_{k,j}(N) := 
\begin{cases}
\BB P_k\left[X_j(N) \geq 1 \: \text{and} \: X_k(N) \geq 1   \right]  ,\quad & k \not= j  \\
\BB P_k\left[X_k(N) \geq 2   \right]  ,\quad & k = j  . 
\end{cases}  
\eqe 
\end{lem}
\begin{proof} 
By positive regularity, there exists $N_0 \in \BB N$ such that 
\eqbn
\BB P_k\left[ X_j(N_0) \geq 1 \right] > 0 ,\quad\forall k,j \in \{1,\dots,d\}  . 
\eqen
By non-singularity, there exists $k_0  \in  \{1,\dots,d\}$ such that each individual of type $k_0$ has a positive chance to have at least two offspring. Now let $k, j \in \{1,\dots,d\}$ be arbitrary. There is a positive $\BB P_k$-probability that we have at least one individual of type $k_0$ in the $N_0$th generation. Conditional on this, there is a positive chance that we have at least two individuals (of any type) in the $N_0+1$th generation. Conditional on this, there is a positive chance that we have at at least one individual of type $k$ and one individual of type $j$ (if $j\not=k$) or at least two individuals of type $k$ (if $k=j$) in the $2N_0+1$th generation. Thus the lemma statement holds with $N  =2N_0 +1$. 
\end{proof}

Finally, we need the following statement, which follows from the hypotheses of our main theorem.

\begin{lem} \label{lem-gf-bded}
Assume the hypotheses of Theorem~\ref{thm-fixed-gf}. 
There exists $q^0 \in (1,\infty)^d$ such that 
\eqbn
f_k(q^0) < q^0_k ,\quad\forall k\in \{1,\dots,d\}. 
\eqen
\end{lem}
\begin{proof}
Recall the matrix $M$ from~\eqref{eqn-mean-matrix}. Let $\rho$ be the largest eigenvalue of $M$, which is strictly less than one by our subcriticality hypothesis~\ref{item-gf-extinction}. Also let $v$ be the corresponding eigenvector, normalized so that $|v|=1$. Since we have assumed that our branching process is positive regular and non-singular, the Perron-Frobenius theorem implies that each of the $d$ components of $v$ is positive. 

Write $\bd 1 = (1,\dots,1)$ for the vector of $d$ ones. By our continuity hypothesis~\ref{item-gf-cont} and since $f(\bd 1) = \bd 1$, the function $f$ is finite on a neighborhood of $\bd 1$. Furthermore, $M$ is the Jacobian matrix of $f$ at $\bd 1$. By Taylor expansion, for $\ep > 0$, we have 
\eqbn
f(\bd{1} + \ep v) = \bd{1} + \rho \ep v + O(\ep^2) ,\quad \text{as $\ep \to 0$}. 
\eqen 
Since $\rho < 1$, this implies the lemma statement with $q^0 = \bd 1 + \ep v$ for a small enough choice of $\ep>0$. 
\end{proof}

\subsection{Monotonicity properties}
\label{sec-mono}

Throughout the rest of this section, we assume that we are in the setting of Theorem~\ref{thm-fixed-gf}. 
In this section, we will establish a number of monotonicity properties for the generating function which will play an important role in our proofs. 
 
Instead of working with the generating function $f$, for most of the proof it is more convenient to work with the closely related function $g$ defined as follows. 

\begin{defn} \label{def-log-gf}
For $\lambda = (\lambda_1,\dots,\lambda_d ) \in (-\infty,\infty]^d$, define 
\eqbn
g_k(\lambda) := \log f_k(e^{\lambda_1} , \dots, e^{\lambda_d} ) \quad \text{and} \quad 
g(\lambda) := (g_1(\lambda),\dots,g_d(\lambda)) .
\eqen
We also let $g^{(n)}$ be the composition of $g$ with itself $n$ times. 
Then, using~\eqref{eqn-gen-func-n}, we see that the $k$th component of $g^{(n)}$ is given by
\eqb \label{eqn-convex-func-n} 
g_k^{(n)}(\lambda) 
= \log f_k^{(n)}(e^{\lambda_1},\dots,e^{\lambda_d})
= \log \BB E_k\left[ \prod_{j=1}^d e^{\lambda_j X_j(n)} \right] .
\eqe 
\end{defn}

We observe that 
\eqb \label{eqn-fixed-pt-equiv}
g(\lambda_1,\dots,\lambda_d) = (\lambda_1,\dots,\lambda_d) \quad \text{if and only if} \quad
f(e^{\lambda_1},\dots,e^{\lambda_d}) = (e^{\lambda_1},\dots, e^{\lambda_d}) .
\eqe
Therefore, Theorem~\ref{thm-fixed-gf} is equivalent to the following statement.

\begin{thm} \label{thm-fixed}
There exists $\lambda \in (0,\infty)^d $ such that $g(\lambda) = \lambda$.
\end{thm}

Although it will not be needed for our proofs, we remark that it is easy to see from H\"older's inequality that each $g_k$ is convex.

We will make frequent use of the following notation for componentwise inequalities.

\begin{defn} \label{def-componentwise}
For $\lambda,\eta\in (-\infty,\infty]^d$, we write $\lambda \leq \eta$ if $\lambda_j \leq \eta_j$ for each $j \in \{1,\dots,d\}$. We write $\lambda < \eta$ if $\lambda \leq \eta$ and there exists $j\in \{1,\dots,d\}$ such that $\lambda_j < \eta_j$. 
\end{defn}

The following elementary fact will play a fundamental role in our proofs. 

\begin{lem}[Monotonicity] \label{lem-mono}
Let $\lambda,\eta \in [0,\infty)^d$ such that $\lambda \leq \eta$.
Then, with $g^{(n)}$ as in Definition~\ref{def-log-gf}, we have $g_k^{(n)}(\lambda) \leq g_k^{(n)}(\eta)$ for each $k \in \{1,\dots,d\}$ and each $n\in\BB N$.
Moreover, there exists $N\in\BB N$ such that if $\lambda < \eta$, then in fact $g_k^{(n)}(\lambda) < g_k^{(n)}(\eta)$ for each $k\in \{1,\dots,d\}$ and each $n \geq N$ such that $|g^{(n)}(\eta)|<\infty$. 
\end{lem}
\begin{proof}
The statement that $g_k^{(n)}(\lambda) \leq g_k^{(n)}(\eta)$ for each $k \in \{1,\dots,d\}$ is immediate from the definition of $g^{(n)}$. 

To get the statement with strict inequality, let $j_* \in \{1,\dots,d\}$ be such that $\lambda_{j_*} < \eta_{j_*}$. By positive regularity (Definition~\ref{def-bp-basic}), there exists $N\in\BB N$ such that for each $k , j\in \{1,\dots,d\}$, we have $\BB P_k\left[ X_j(N) \geq 1 \right] > 0$. From this and the Markov property of the branching process, in fact $\BB P_k\left[ X_j(n) \geq 1 \right] > 0$ for each $n \geq N$ and each $k,j\in \{1,\dots,d\}$. 
Therefore, for $n\geq N$, it holds with positive $\BB P_k$-probability that $\prod_{j=1}^d e^{\lambda_j X_j(n)}  < \prod_{j=1}^d e^{\eta_j X_j(n)}$. By taking $\BB E_k$-expectations, we get $g_k^{(n)}(\lambda) < g_k^{(n)}(\eta)$, provided $g_k^{(n)}(\eta)$ is finite.  
\end{proof}

We next show that the set of points $\lambda \in [0,\infty)^d\setminus \{0\}$ for which $|g^{(n)}(\lambda)|$ remains bounded away from 0 and $\infty$ as $n\to\infty$ is ``unstable'', in the sense that slightly increasing or decreasing the components of a point in this set will take us out of the set.

\begin{prop} \label{prop-unstable}
\begin{enumerate}[$(i)$]
\item Let $\lambda \in [0,\infty)^d \setminus \{0\}$ and assume that $\limsup_{n\to\infty} | g^{(n)}(\lambda)| < \infty$. 
If $\eta \in [0,\infty)^d$ such that $\eta < \lambda$, then $\lim_{n\to\infty} |g^{(n)}(\eta)| = 0$.  \label{item-unstable-zero}
\item Let $\lambda \in [0,\infty)^d \setminus \{0\}$ and assume that $\limsup_{n\to\infty} | g^{(n)}(\lambda)| > 0 $. 
If $\eta \in [0,\infty)^d$ such that $\eta > \lambda$, then $\limsup_{n\to\infty} |g^{(n)}(\eta)| = \infty$. \label{item-unstable-infty} 
\end{enumerate}
\end{prop}

The proof of Proposition~\ref{prop-unstable} is based on the following lemma, which is a straightforward consequence of the Vitali convergence theorem.

\begin{lem} \label{lem-unstable-scaled}
Let $\lambda \in [0,\infty)^d \setminus \{0\}$ and assume that $\limsup_{n\to\infty} | g^{(n)}(\lambda)| < \infty$. 
Then for any $r\in [0,1)$, we have $\lim_{n\to\infty} |g^{(n)}(r\lambda)| = 0$.  
\end{lem}
\begin{proof}
Write $Z_n := \prod_{j=1}^d e^{r \lambda_j X_j(n)}$, so that for $k\in\{1,\dots,d\}$, 
\eqbn
g_k^{(n)}(r \lambda) = \log \BB E_k[Z_n] .
\eqen
Since the extinction probability is one (Hypothesis~\ref{item-gf-extinction}), a.s.\ each $X_j(n)$ is equal to zero for all sufficiently large $n$, so a.s.\ $\lim_{n\to\infty} Z_n = 1$. 
By H\"older's inequality, 
\alb
\BB E_k[Z_n \BB 1_{(Z_n > 1)} ]  
\leq \BB E_k[Z_n^{1/r}]^r \BB P_k[Z_n > 1]^{1-r}  
= \exp\left(r g_k^{(n)}(\lambda) \right) \BB P_k\left[ Z_n > 1 \right]^{1-r} .
\ale
Since $\lim_{n\to\infty} Z_n = 1$ a.s.\ and $\limsup_{n\to\infty} g_k^{(n)}(\lambda) < \infty$, this goes to zero as $n\to\infty$. By the Vitali convergence theorem, we get that $\lim_{k\to\infty} \BB E_k[Z_n] = 1$, which implies that $\lim_{k\to\infty} g_k^{(n)}(r\lambda) = 0$.
\end{proof}

\begin{proof}[Proof of Proposition~\ref{prop-unstable}]
First assume that $\eta_j < \lambda_j$ for each $j\in \{1,\dots,d\}$. 
Then there exists $r\in (0,1)$ such that $  \eta \leq r \lambda$. 
By Lemma~\ref{lem-unstable-scaled}, this implies that $\lim_{n\to\infty} |g^{(n)}(\eta)| = 0$. 

More generally, assume that we only have $\eta < \lambda$, i.e., $\eta_j \leq \lambda_j$ for each $j\in \{1,\dots,d\}$ and at least one of the inequalities is strict. 
By the second assertion of Lemma~\ref{lem-mono}, there exists $N\in\BB N$ such that $g_k^{(N)}(\eta) < g_k^{(N)}(\lambda)$ for each $k \in \{1,\dots,d\}$. Since $\limsup_{n\to\infty} |g^{(n)}(g^{(N)}(\lambda))| < \infty$, we can apply the first paragraph with $g^{(N)}(\lambda)$ in place of $\lambda$ and $g^{(N)}(\eta)$ in place of $\eta$ to get that $\limsup_{n\to\infty} |g^{(n)}(\eta)| = 0$. 
This gives assertion~\eqref{item-unstable-zero}.
Assertion~\eqref{item-unstable-infty} follows by applying Assertion~\eqref{item-unstable-zero} with the roles of $\lambda$ and $\eta$ interchanged.
\end{proof} 

We have the following minor generalization of Proposition~\ref{prop-unstable}. 

\begin{lem} \label{lem-unstable-ssl}
Fix $m\in\BB N$. 
\begin{enumerate}[$(i)$]
\item Let $\lambda \in [0,\infty)^d \setminus \{0\}$ and assume that $\limsup_{r \to\infty} | g^{(r m )}(\lambda)| < \infty$. 
If $\eta \in [0,\infty)^d$ such that $\eta < \lambda$, then $\lim_{r \to\infty} |g^{(r m)}(\eta)| = 0$.  \label{item-unstable-zero-ssl}
\item Let $\lambda \in [0,\infty)^d \setminus \{0\}$ and assume that $\limsup_{r \to\infty} | g^{(r m)}(\lambda)| > 0 $. 
If $\eta \in [0,\infty)^d$ such that $\eta > \lambda$, then $\limsup_{r \to\infty} |g^{(r m)}(\eta)| = \infty$. \label{item-unstable-infty-ssl} 
\end{enumerate}
\end{lem}
\begin{proof}
Consider the branching process $\{\wt X(r)\}_{r\in\BB N_0}$ where we observe the population only once every $m$ generations. That is, we set $\wt X_k(r) = X_k(m r)$ for each $k \in \BB N$. One easily verifies that $\{\wt X(r)\}_{r\in\BB N_0}$ satisfies the hypotheses of Theorem~\ref{thm-fixed-gf}. The lemma follows by applying Proposition~\ref{prop-unstable} to this branching process
\end{proof}

Another useful consequence of Proposition~\ref{prop-unstable} is the following lemma. 

\begin{lem} \label{lem-monotone-lim}
Let $\lambda \in [0,\infty)^d$.
If there exists $ m \in \BB N$ such that $g^{(m)}(\lambda) < \lambda$, then $\lim_{n\to\infty} g^{(n)}(\lambda) = 0$. 
Moreover, if there exists $m \in\BB N$ such that $g^{(m)}(\lambda) > \lambda$, then $\lim_{n\to\infty} |g^{(n)}(\lambda)| = \infty$.
\end{lem}
\begin{proof}
Let $m \in \BB N$ such that $g^{(m)}(\lambda) < \lambda$. 
By Lemma~\ref{lem-mono}, we have
\eqb  \label{eqn-shifted-mono} 
g^{(m + n)}(\lambda) \leq g^{(n)}(\lambda) ,\quad \forall n \geq 0 
\eqe 
where here we set $g^{(0)}(\lambda) = \lambda$. Furthermore, the inequality in~\eqref{eqn-shifted-mono} is strict if $n$ is sufficiently large. 
Hence, for each $s \in \{0,\dots,m-1\}$, we have $g^{((r+1) m + s)}(\lambda) \leq g^{(r m +s)}(\lambda)$ for each integer $r  \geq 0$. This implies that for each such $s$, the limit
\eqb  \label{eqn-mono-ssl}
\ol\lambda^s := \lim_{r \to \infty} g^{(r m + s)}(\lambda) 
\eqe 
exists. 

We claim that $\ol\lambda^s = 0$.  
Assume for contradiction that $\ol\lambda^s  > 0$ (i.e., at least one component is greater than zero). 
We have $g^{(r m)}(\ol\lambda^s) = \ol\lambda^s$ for each $r \in \BB N$, so in particular $\limsup_{r\to\infty} | g^{(r m  )}( \ol\lambda^s)| > 0$. Since the inequality in~\eqref{eqn-shifted-mono} is strict for large enough $n$, we have $\ol\lambda^s < g^{(s)}(\lambda)$. 
Hence Lemma~\ref{lem-unstable-ssl}\eqref{item-unstable-infty-ssl} implies that $\limsup_{r\to\infty} |g^{(r m)}( g^{(s)}(\lambda) ) | = \infty$, contrary to~\eqref{eqn-mono-ssl}. 
Thus $\lim_{r\to\infty} g^{(r m+s)}(\lambda) = 0$ for every $s \in \{0,\dots,m-1\}$, which implies that $\lim_{n\to\infty} g^{(n)}(\lambda) = 0$. 
  
The statement when $g^{(m)}(\lambda) > \lambda$ is proven similarly.  
\end{proof}

\subsection{Basins of attraction of 0 and $\infty$}
\label{sec-basins}

Our proof of Theorem~\ref{thm-fixed} is based on splitting up space based on the large-$n$ behavior of $g^{(n)}$. 
As in Section~\ref{sec-outline}, we define the following sets: 
\eqb \label{eqn-basin-def}
S_0 := \left\{\lambda \in [0,\infty]^d : \lim_{n\to\infty} |g^{(n)}(\lambda)| = 0\right\} \quad \text{and} \quad 
S_\infty := \left\{ \lambda \in [0,\infty]^d : \limsup_{n\to\infty} |g^{(n)}(\lambda)| = \infty \right\} .
\eqe 
Note the asymmetry in the definitions: we have a $\lim$ in the definition of $S_0$, but a $\limsup$ in the definition of $S_\infty$. The reason for this asymmetry comes from Proposition~\ref{prop-unstable}. 
We also note that since our branching process is positive regular (Definition~\ref{def-componentwise}), there exists $N\in\BB N$ such that if any of the $d$ coordinates of $\lambda$ is equal to $\infty$, then $g^{(N)}(\lambda) = (\infty,\dots,\infty)$. This implies that
\eqb \label{eqn-infty-subset}
[0,\infty]^d\setminus [0,\infty)^d \subset S_\infty.
\eqe 
    
\begin{lem} \label{lem-basin-closed}
We have $g(S_0)\subset S_0$, $g(S_\infty) \subset S_\infty$, $g^{-1}(S_0)\subset S_0$, and $g^{-1}(S_\infty) \subset S_\infty$. 
\end{lem}
\begin{proof}
This is immediate from the definition~\eqref{eqn-basin-def}. 
\end{proof}

Our next goal is to prove that the sets $S_0$ and $S_\infty$ are open with respect to the topology of $[0,\infty)^d$ (see Lemma~\ref{lem-basin-open} below). 
To do this, we will first show that $0$ and $\infty$ are attractive fixed points of $g$, in the sense that $S_0$ and $S_\infty$ contain neighborhoods of $0$ and $\infty$, respectively, in $[0,\infty]^d$.

\begin{lem} \label{lem-attractive-zero}
There exists $\ep > 0$ such that for each $\lambda \in [0,\ep]^d$, we have $\lim_{n\to\infty} |g^{(n)}(\lambda)| = 0$. 
\end{lem}
\begin{proof}
By Lemma~\ref{lem-gf-bded}, there exists $\lambda^0 = (\lambda_1^0,\dots,\lambda_d^0) \in (0,\infty)^d$ such that $g(\lambda^0) < \lambda^0$. By Lemma~\ref{lem-monotone-lim}, this implies that $\lim_{n\to\infty} |g^{(n)}(\lambda^0)| = 0$. 
By the monotonicity of $g$ (Lemma~\ref{lem-mono}), the lemma statement holds with
\eqbn
\ep := \min_{j\in \{1,\dots,d\}} \lambda_j^0 > 0 . 
\eqen 
\end{proof}

We next need a complementary statement for the behavior of $g^{(n)}(\lambda)$ when $|\lambda|$ is large. 

\begin{lem} \label{lem-attractive-infty}
There exists $C > 0$ such that the following is true. 
If $\lambda \in [0,\infty)^d \setminus [0,C]^d$, then $g(\lambda) > \lambda$ and $\lim_{n\to\infty} |g^{(n)}(\lambda)| = \infty$. 
\end{lem} 
\begin{proof} 
Let $N\in\BB N$ and $p_{k,j}(N)$ for $k,j \in \{1,\dots,d\}$ be as in Lemma~\ref{lem-gf-pos}. 
By Lemma~\ref{lem-gf-pos}, each $p_{k,j}(N)$ is positive.  
Therefore,
\eqbn
 m := \min\{p_{k,j}(N) : k,j \in \{1,\dots,d\} \}  > 0 .
\eqen
For $\lambda \in [0,\infty)^d$, the definitions of $p_{k,j}(N)$ and $m$ imply that 
\alb
\exp( g_k^{(N)}(\lambda) ) 
= \BB E_k\left[ \prod_{j=1}^d e^{\lambda_j X_j(N)} \right] 
\geq \min_{j \in \{1,\dots,d\}} \left(  e^{ \lambda_j + \lambda_k}     p_{k,j}(N) \right)  
\geq e^{\lambda_k }  m \exp\left(  \min_{j\in\{1,\dots,d\}}  \lambda_j \right) .
\ale
Therefore, if $\min_{j \in \{1,\dots,d\}} \lambda_j \geq C := \log m^{-1}$, then $g^{(N)}(\lambda) > \lambda$. By Lemma~\ref{lem-monotone-lim}, this implies that $\lim_{n\to\infty} |g^{(n)}(\lambda)| = \infty$. 
\end{proof}

\begin{lem} \label{lem-basin-open} 
Define $S_0$ and $S_\infty$ as in~\eqref{eqn-basin-def}. Then $S_0$ and $S_\infty$ are relatively open subsets of $[0,\infty]^d$. 
\end{lem}
\begin{proof}
Let $\ep > 0$ be as in Lemma~\ref{lem-attractive-zero}, so that 
\eqb \label{eqn-use-attractive-zero}
[0,\ep]^d \subset S_0 . 
\eqe
Let $\lambda \in S_0$. 
By the definition of $S_0$, there exists $n\in\BB N$ such that $g^{(n)}(\lambda) \in [0,\ep)^d$. 
By Hypothesis~\ref{item-gf-cont}, the function $g^{(n)}$ is continuous from $[0,\infty]^d$ to itself. 
Since $[0,\ep)^d$ is a relatively open subset of $[0, \infty]^d$, the set $(g^{(n)})^{-1}([0,\ep)^d)$ is a neighborhood of $\lambda$ in $[0,\infty]^d$. 
By~\eqref{eqn-use-attractive-zero} and Lemma \ref{lem-basin-closed}, $(g^{(n)})^{-1}([0,\ep)^d) \subset S_0$. 
Thus, $S_0$ is relatively open. 

To prove the statement for $S_\infty$, let $C > 0$ be as in Lemma~\ref{lem-attractive-infty}. By that lemma and~\eqref{eqn-infty-subset}, $[0,\infty]^d \setminus [0,C]^d \subset S_\infty$. Since $[0,\infty]^d \setminus [0,C]^d$ is a relatively open subset of $[0,\infty]^d$ and $g^{(n)} : [0,\infty]^d \to [0,\infty]^d$ is continuous, we can use exactly the same argument as in the case of $S_0$ to conclude that $S_\infty$ is relatively open. 
\end{proof}

Our next goal is to show that $[0,\infty)^d \setminus (S_0\cup S_\infty)$ is a retract (in fact, a deformation retract) of an open subset of $(0,\infty)^d$ with the topology of a closed disk. The key observation for this purpose is the following lemma.

\begin{lem} \label{lem-ray-intersect} 
Let $u \in [0,\infty)^d$ with $|u| = 1$. There exists exactly one value of $r \in (0,\infty)$ such that $r u \in [0,\infty)^d \setminus (S_0 \cup S_\infty)$. 
\end{lem}
\begin{proof}
Let $L = \{r u  : r  > 0\}$ be the ray started at 0 which passes through $u$. By~\eqref{eqn-basin-def}, the sets $S_0$ and $S_\infty$ are disjoint. By Lemma~\ref{lem-basin-open}, the intersections of these sets with $[0,\infty)^d$ are open subsets of $[0,\infty)^d$. By Lemmas~\ref{lem-attractive-zero} and~\ref{lem-attractive-infty}, $L$ intersects both $S_0$ and $S_\infty$.  
Since $L \subset [0,\infty)^d$ is connected, it follows that $L$ must contain at least one element of $[0,\infty)^d \setminus (S_0\cup S_\infty)$. 

Suppose by way of contradiction that $L$ contains at least two elements of $[0,\infty)^d \setminus (S_0\cup S_\infty)$. Then there exists $R > r > 0$ such that $R u , ru \in [0,\infty)^d \setminus (S_0 \cup S_\infty)$. By the definition~\eqref{eqn-basin-def} of $S_0$ and $S_\infty$, we have $\limsup_{n\to\infty} |g^{(n)}(r u)| > 0$ and $\limsup_{n\to\infty} |g^{(n)}(R u)| < \infty$. This contradicts Lemma~\ref{lem-unstable-scaled}. 
\end{proof}

\begin{figure}[ht!]
\begin{center}
\includegraphics[width=0.5\textwidth]{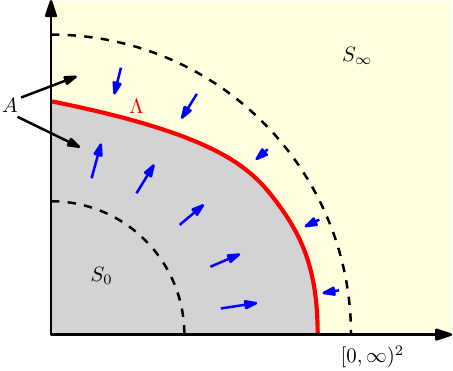}  
\caption{\label{fig-brouwer-proof} Illustration of the proof of Lemma~\ref{lem-retract} in the case $d=2$. We construct a deformation retraction from the region $A$ (bounded by the two dashed curves) to the red set $\Lambda = [0,\infty)^2 \setminus (S_0\cup S_\infty)$ by sliding points along rays through the origin. 
}
\end{center}
\end{figure}

\begin{lem} \label{lem-retract}
There exists $R > r > 0$ such that the set
\eqb \label{eqn-retract-annulus} 
A = A(r,R) := \left\{ \lambda \in [0,\infty)^d  : r \leq |\lambda| \leq R \right\} 
\eqe 
deformation retracts onto $[0,\infty)^d \setminus (S_0\cup S_\infty)$. 
\end{lem}
\begin{proof}  
To lighten notation, we define
\eqbn
\Lambda := [0,\infty)^d \setminus (S_0 \cup S_\infty) 
\quad \text{and} \quad 
B := \left\{ \lambda \in [0,\infty)^d : |\lambda| = 1 \right\} .
\eqen 
By Lemmas~\ref{lem-attractive-zero} and~\ref{lem-attractive-infty}, there exists $R > r > 0$ such that $\Lambda \subset A$, with $A$ as in~\eqref{eqn-retract-annulus}. Henceforth fix such a choice of $r$ and $R$. 

For $\lambda \in [0,\infty)^d \setminus \{0\}$, let $\pi(\lambda) := \lambda/|\lambda|$ be the radial projection onto $B$. 
By Lemma~\ref{lem-ray-intersect}, $\pi|_\Lambda$ is a bijection from $\Lambda$ to $B$. 
Obviously, $\pi$ is continuous. 
By Lemma~\ref{lem-basin-open}, the set $\Lambda$ is a relatively closed subset of $[0,\infty]^d$. By Lemma~\ref{lem-attractive-infty}, $\Lambda$ is also bounded. Hence $\Lambda$ is compact. 
Therefore, $\pi|_\Lambda$ is a homeomorphism. 
Hence, there exists a continuous inverse $\phi : B\to \Lambda$ such that
\eqbn
\phi(u) / |\phi(u)| = u ,\quad \forall u \in B .
\eqen 

For $\lambda \in A$ and $t\in [0,1]$, define 
\eqb \label{eqn-retract-def}
H_t(\lambda) :=  (1-t)  \lambda + t \phi\left( \frac{\lambda}{|\lambda|} \right) .
\eqe 
We claim that $H$ is a deformation retraction from $A$ to $\Lambda$. 
Since $\phi$ is continuous, $H$ is jointly continuous in $t$ and $\lambda$. 
By construction, for each $\lambda \in \Lambda$, we have $\phi(\lambda/|\lambda|) = \lambda$. 
Therefore, $H_t(\lambda) = \lambda$ for each $t\in [0,1]$ and $\lambda\in \Lambda$.
Furthermore, for each $\lambda \in A$, the points $\lambda$ and $\phi(\lambda/|\lambda|)$ lie on the same ray from 0 to $\infty$. Since $\Lambda \subset A$, this gives
\eqbn
|H_t(\lambda)| = (1-t) |\lambda| + t |\phi(\lambda/|\lambda|)|  \in [r,R]  
\eqen
so each $H_t$ maps $A$ to itself.  
Finally, for each $\lambda\in A$ we have $H_0(\lambda) = \lambda$ and $H_1(\lambda) = \phi(\lambda/|\lambda|) \in \Lambda$. Thus $H$ is a deformation retraction from $A$ to $\Lambda$, as required.
\end{proof}

\begin{proof}[Proof of Theorem~\ref{thm-fixed}]
To lighten notation, as above we write $\Lambda := [0,\infty)^d \setminus (S_0 \cup S_\infty)$. 
We first argue that $g(\Lambda) \subset \Lambda$. Indeed, if $\lambda\in [0,\infty)^d$ and $g(\lambda ) \in S_0$, then Lemma~\ref{lem-basin-closed} implies that $\lambda \in S_0$. Hence, $\lambda \notin \Lambda$. Similarly, if $g(\lambda) \in S_\infty$, then $\lambda \notin \Lambda$. Hence, if $\lambda\in\Lambda$, we must have $g(\lambda)\in\Lambda$.  

Let $A$ be as in Lemma~\ref{lem-retract}, so that $\Lambda\subset A$. The set $A$ is homeomorphic to a closed $d$-dimensional Euclidean ball. 
By the Brouwer fixed point theorem, every continuous function from $A$ to itself has a fixed point. 

Let $H_1 : A \to \Lambda$ be a retraction, as given by Lemma~\ref{lem-retract}. 
Since $g(\Lambda)\subset\Lambda$, we have that $g\circ H_1 : A \to \Lambda$. In particular, $g\circ H_1$ is a continuous function from $A$ to itself, so there exists $\lambda \in A$ such that $g(H_1(\lambda)) = \lambda$. Necessarily, $\lambda\in\Lambda$, so $H_1(\lambda) = \lambda$ since $H_1$ is a retraction. Thus $g(\lambda) = \lambda$. 

Since $0\notin \Lambda$, we have $\lambda \in [0,\infty)^d \setminus \{0\}$. By Lemma~\ref{lem-fixed-pos}, in fact $\lambda \in (0,\infty)^d$.  
\end{proof}

\section{The countable-type case}
\label{sec-infty}

One can prove versions of Theorem~\ref{thm-fixed-gf} for branching processes with countably many types by, roughly speaking, applying Theorem~\ref{thm-fixed-gf} and taking a limit as the number of possible types goes to $\infty$. To illustrate this, we will prove one version of Theorem~\ref{thm-fixed-gf} in the countable-type case (Theorem~\ref{thm-fixed-pt-infty} below). However, it seems likely to us that there is not a single, elegant theorem statement which covers all countable-type branching process which one might want to consider simultaneously. 

\subsection{Countable-type branching processes} 
\label{sec-ctble-type}
Consider a multi-type branching process $X = \{X(n)\}_{n \in \BB N_0}$ with types indexed by $k\in \BB N$. 
We assume that each individual has only finitely many offspring, i.e., for each $k \in \BB N$, 
\eqb \label{eqn-finite-offspring} 
\BB P_k \left[ \sum_{j=1}^\infty X_j(1)  < \infty \right] = 1 .
\eqe 
For each $k\in \BB N$, we define $f : (0,\infty]^{\BB N  } \to [0,\infty]$ by 
\eqbn
f_k(q)  = \BB E_k\left[\prod_{j=1}^\infty q_j^{X_j(1)} \right] ,\quad \text{where} \quad q = \{q_j\}_{j\in\BB N }  
\eqen
(with the same conventions about $\infty$ as in~\eqref{eqn-generating-function}) and we define the \textbf{generating function} \eqbn
f(q) := \{f_j(q)\}_{j\in\BB N} \in [0,\infty]^{\BB N} . 
\eqen
We also define $f^{(n)}$ to be the composition of $f$ with itself $n$ times. 
Extending $d$ to $\BB N$, the definitions of positive regularity and non-singularity in Definition \ref{def-bp-basic} also apply to the countable-type branching process $X$. 

In contrast to the finite-type case, when there are infinitely many types, the extinction probability is not uniquely defined. There are three types of extinction probabilities: (1) \textbf{global extinction},  defined as $\lim_{n\to\infty} \sum_{k=1}^\infty X_k(n) = 0$, (2)  \textbf{partial extinction}, defined as $\lim_{n \to \infty} X_k(n)=0$ for all $k \in \BB N$, and (3) \textbf{local extinction} in a fixed subset of types $A \subset {\BB N}$, defined as $\lim_{n\to\infty} \sum_{k\in A} X_k(n)=0$. Let $q$, $\tilde{q}$ and $q(A)$ be the probabilities of global extinction, partial extinction and local extinction in $A$ respectively. It can be shown that if $X$ is irreducible and non-singular, then $q\leq q(A)\leq \tilde{q} \leq 1$, and all of the extinction probabilities are solutions to the fixed point equation $s=f(s)$. Let $S$ be the collection of all the fixed points in $[0,1]^{\BB N}$ and $Q$ be the collection of all the extinction probabilities. Unlike the finite-type case where $S$ has at most two elements, countable-type branching process can have a continuum of fixed points. Indeed,  \cite{Bertacchi2014, Bertacchi2017, Braunsteins2019} showed that there exist countable-type branching processes with a finite number of elements in $Q$ and a continuum of fixed points in $S$.  \cite{Bertacchi2020} gave an example with uncountably many elements in both $Q$ and $S$. Some of these references are for branching random walks with an at most countable set of particle locations. In this setting, a multi-type branching process is equivalent to a branching random walk, where the type of a particle corresponds to its location.

If $X$ is  irreducible and non-singular, it is proved in \cite[Theorem 3.1]{Moyal1964} that $q$ is the component-wise minimal element of $S$, and \cite[Theorem 3.1]{Bertacchi2022} that $\tilde{q}$ is the component-wise largest or second largest element of $S$ depending on whether $\tilde{q}=1$ or $\tilde{q}<1$. This implies that the concept of global extinction parallels that of extinction in the finite-type case. In what follows, we will focus on global extinction exclusively.

As in the finite-type case, conditioning an infinite-type branching process on global extinction gives an infinite-type branching process with extinction probability one~\cite{jl-cond-on-ext}. If the global extinction probability for $X$ is one, we can define a \textbf{conjugate branching process} for $X$ in exactly the same manner as in the finite-type case (Definition~\ref{def-conjugate}). Moreover, just as in the finite-type case, conjugate branching processes correspond to solutions of $f(q) =q$ in $(1,\infty)^{\BB N}$. Indeed, since the result of~\cite{jl-cond-on-ext} applies also in the countable-type case,~\eqref{eqn-cond-on-ext} in Proposition \ref{prop-cond-on-ext} still holds. Replacing extinction with global extinction, the statement of Proposition~\ref{prop-conjugate} extends essentially verbatim to the countable-type case, with the same proof. We remark that in the countable-type case, the conjugate process has global extinction probability less than 1, but the partial extinction probability may be 1. 
 
We can approximate an infinite-type branching process by finite-type branching processes in the following manner.  
For $d\in \BB N$, let $X^d = \{X^d(n)\}_{n\in\BB N_0}$ be the $d$-type branching process obtained from $X$ by immediately killing each individual of type $k \geq d+1$. In other words, an individual in $X$ corresponds to an individual in $X^d$ if and only if each of its ancestors has type $k \leq d$. We call $X^d$ the \textbf{branching process truncated at level $d$}, and we denote all objects associated with $X^d$ by an additional superscript $d$. Note that $X^d_k(1)=0$ if $k \geq d+1$.
Consequently, the generating function $f^d = (f_1^d,\dots,f_d^d)$ for $X^d$ is given by
\eqb  \label{eqn-gf-truncated}
f^d_k(q_1,\dots,q_d) 
= \BB E_k\left[ \prod_{j=1}^d q_j^{X_j^d(1)} \right] 
= f_k(q_1,\dots,q_d,1,1,\dots)
,\quad\forall k \in \{1,\dots,d\} .
\eqe 
The idea of approximating a countable type branching process in this manner also appears, e.g., in~\cite{hln-ext-prob}.
 
\subsection{Existence of a conjugate branching process}

Let $X = \{X(n)\}_{n\in\BB N_0}$ be a countable type branching process with global extinction probability one. 
Let us now state a list of assumptions which ensure the existence of a conjugate branching process, equivalently, a fixed point for the generating function $f$ in $(1,\infty)^{\BB N_0}$.
\begin{enumerate}[I.] 
\item \textbf{(Fixed points for truncated processes).} For each $d\in\BB N$, the truncated branching process $X^d$ satisfies the hypotheses of Theorem~\ref{thm-fixed-gf}. \label{item-gf-trunc-infty} 
\item \textbf{(Bounded growth in $k$).} There exist functions
\eqb \label{eqn-weight-function}
\phi , \psi : \BB N \to [1,\infty) \quad \text{such that} 
\quad \lim_{k \to\infty} \phi(k) = \infty \quad \text{and} \quad 
\lim_{k\to\infty} \frac{ \log\psi(k)}{ \phi(k) } = 0
\eqe
such that for each $k\in\BB N$, \label{item-gf-growth-infty}
\eqb
f_k(q) \leq \psi(k) [f_1(q)]^{\phi(k)} ,\quad\forall q \in [1,\infty)^{\BB N}.
\eqe
\item \textbf{(Subcriticality with weights).} There exists $s > 1$ such that \label{item-gf-bded-infty}
\eqbn
\BB E_k\left[ s^{\sum_{j=1}^\infty \phi(j) X_j(1)} \right] \leq s^{\phi(k)} , \quad \forall k \in \BB N . 
\eqen
\item \textbf{(Finiteness of $f_k$).} There exists $t > 1$ such that  \label{item-gf-perturb-infty} 
\eqbn
f_1(t,1,1,\dots) > t \quad \text{and} \quad
f_k( \{\psi(j) t^{\phi(j)}\}_{j \in \BB N} )  < \infty , \quad \forall k \in \BB N. 
\eqen  
\end{enumerate}
We think of $\phi(k)$ from~\eqref{eqn-weight-function} as the ``weight'' of an individual of type $k$. 
For various subcritical branching processes arising from loop-decorated random planar maps, as discussed in Section~\ref{sec-lqg}, we expect something similar to the above list of hypotheses to be satisfied with $\phi(k)$ taken to be a constant multiple of the loop length $k$. See Example~\ref{example-ctble} below for a simple class of examples where the above hypotheses are satisfied. 
 
\begin{thm} \label{thm-fixed-pt-infty}
Under the above assumptions, there exists $q \in (1,\infty)^{\BB N }$ such that $f(q) = q$.
\end{thm}

We prove Theorem~\ref{thm-fixed-pt-infty} by applying Theorem~\ref{thm-fixed-gf} and taking a (subsequential) limit.
Recall the truncated branching process $X^d$ and its generating function $f^d$ from Section~\ref{sec-ctble-type}. 
As in Definition~\ref{def-log-gf}, we define
\eqb \label{eqn-log-gf-infty}
g(\lambda) := \log f(\exp(\lambda)) ,\quad\forall \lambda \in (-\infty,\infty]^{\BB N_0} \quad \text{and} \quad 
g^d(\lambda) := \log f^d(\exp(\lambda)) ,\quad\forall \lambda (-\infty,\infty]^d
\eqe 
where here $\log$ and $\exp$ are applied componentwise. To prove Theorem~\ref{thm-fixed-pt-infty}, it suffices to find $\lambda \in (0,\infty)^{\BB N_0}$ such that $g(\lambda) = \lambda$ (we can then set $q = \exp(\lambda)$). 
 
By Theorem~\ref{thm-fixed-gf}, for each $d\in\BB N$ there exists 
\eqb \label{eqn-fixed-pt-trunc}
\lambda^d = (\lambda^d_1,\dots,\lambda^d_d) \in (0,\infty)^d \quad \text{such that} \quad g^d(\lambda^d) = \lambda^d .
\eqe 
If there is more than one fixed point for $g^d$ in $(0,\infty)^d$, we choose one in an arbitrary manner. The following lemma tells us that for each fixed $k\in\BB N$, the sequence $\{\lambda_k^d\}_{d\in\BB N}$ is pre-compact. 

\begin{lem} \label{lem-fixed-limsup}
For each $k \in \BB N$, we have $\limsup_{d\to\infty} \lambda_k^d < \infty$. 
\end{lem}
\begin{proof}
Fix $k\in\BB N$. 
Since each $X^d$ satisfies the hypotheses of Theorem~\ref{thm-fixed-gf}, we can apply Lemma~\ref{lem-gf-pos} to $X^k$ to find that there exists $N\in\BB N$ (depending on $k$) such that
\eqbn
p_k := \BB P_k\left[ X_k^k(N) \geq 2 \right]  > 0 . 
\eqen
Note that by the definition of $X^d$, we also have
\eqbn
  \BB P_k\left[ X_k^d(N) \geq 2 \right]  \geq p_k  ,\quad\forall d\geq k. 
\eqen
Since $g^{d,(n)}(\lambda^d) =\lambda^d$ for each $d\geq k$ and $n\geq 1$, 
\alb
\exp( \lambda_k^d ) 
= \exp\left( g_k^{d,(N)}(\lambda^d) \right) 
=  \BB E_k\left[ \exp\left( \sum_{j=1}^d \lambda_j^d X_j^d(N) \right) \right] 
\geq \exp( 2\lambda_k^d) p_k  .
\ale
Re-arranging gives $\lambda_k^d \leq \log\frac{1}{p_k }$.  
\end{proof}

Due to Lemma~\ref{lem-fixed-limsup}, we can take a subsequential limit of each of the sequences $\{\lambda_k^d\}_{d\in\BB N}$ to get a sequence $\lambda  = \{\lambda_k\}_{k\in\BB N} \in [0,\infty)^{\BB N}$ which is a candidate for a fixed point of $g$. 
The hardest part of the proof of Theorem~\ref{thm-fixed-pt-infty} is showing that $\lambda \not= 0$. This is the purpose of the next lemma, and also the place where we use Assumptions~\ref{item-gf-growth-infty} and~\ref{item-gf-bded-infty}.

\begin{lem} \label{lem-fixed-liminf}
There exists $K\in\BB N$ such that 
\eqbn
\liminf_{d  \to \infty} \max_{k\in \{1,\dots,K\}} \lambda_k^d > 0 .
\eqen 
\end{lem} 
\begin{proof}
\noindent\textit{Step 1: lower bound for a $d$-dependent value of $k$.}
We start by establishing a lower bound for $\lambda_k^d$ which holds for \emph{at least one} value of $k \in \{1,\dots,d\}$. This will be done using our strict subcriticality assumption, Assumption~\ref{item-gf-bded-infty}. 
By that assumption, there exists $\delta > 0$ (in particular, $\delta  = \log s$) such that 
\eqb \label{eqn-use-gf-bded-infty}
\BB E_k\left[ \exp\left( \delta   \sum_{j=1}^\infty \phi(j) X_j(1) \right) \right]   
\leq  e^{\delta \phi(k) }   ,\quad\forall k \in \BB N . 
\eqe 
By the definition~\eqref{eqn-log-gf-infty} of $g_k$,~\eqref{eqn-use-gf-bded-infty} is equivalent to
\eqb
g_k\left( \{\delta \phi(j)\}_{j \in \BB N} \right) \leq \delta \phi(k) ,\quad \forall k \in \BB N . 
\eqe
By~\eqref{eqn-gf-truncated}, this implies that
\eqb \label{eqn-gf-truncated-strict} 
g_k^d \left( \delta \phi(1)  ,\dots, \delta \phi(d) \right) \leq  \delta \phi(k)  ,\quad\forall k\in \{1,\dots,d\} .
\eqe 

By the monotonicity of $g^{d,(n)}$ (Lemma~\ref{lem-mono}), the bound~\eqref{eqn-gf-truncated-strict} implies that
\eqb \label{eqn-use-mono-infty} 
  g_k^{d,(n)} \left( \delta \phi(1)  ,\dots, \delta \phi(d) \right) \leq \delta \phi(k) ,\quad\forall k \in \{1,\dots,d\} . 
\eqe  
By Proposition~\ref{prop-unstable}, \eqref{eqn-use-mono-infty} implies that
\eqb \label{eqn-fixed-liminf-zero}
\lim_{n\to\infty} g^{d,(n)}(\eta) = 0,\quad\forall \eta \in [0,\delta\phi(1)) \times \dots \times [0,\delta \phi(d)) .
\eqe 
Since $g^{d,(n)}(\lambda^d) = \lambda^d$ for every $n\in\BB N$, it follows from~\eqref{eqn-fixed-liminf-zero} that for each $d\in\BB N$, 
\eqb \label{eqn-fixed-pt-one-coord}
\exists k(d) \in \{1,\dots,d\}  \quad \text{such that} \quad \lambda_{k(d)}^d \geq \delta \phi(k(d)) .
\eqe

\noindent\textit{Step 2: transferring to a lower bound for a bounded value of $k$.}
If $k(d)$ were bounded above by a constant, then~\eqref{eqn-fixed-pt-one-coord} would yield the lemma statement.
To deal with the possibility that $\limsup_{d\to\infty} k(d)  = \infty$, we will now establish a lower bound for $\lambda_1^d$ in terms of $\lambda_k^d$, for an arbitrary value of $k\in \{1,\dots,d\}$. This is done using our growth assumption, Assumption~\ref{item-gf-growth-infty}. 

For each $d\in\BB N$ and each $k\in \{1,\dots,d\}$, 
\allb \label{eqn-fixed-pt-exp-bd}
\lambda_k^d
&= g_k^d(\lambda^d) \quad \text{(definition of $\lambda^d$)} \notag \\
&= g_k(\lambda_1^d,\dots,\lambda_d^d,1,1,\dots) \quad \text{(by~\eqref{eqn-gf-truncated})} \notag \\
&\leq \log \psi(k) + \phi(k)   g_1(\lambda_1^d,\dots,\lambda_d^d,1,1,\dots) \quad \text{(by Assumption~\ref{item-gf-growth-infty})} \notag \\
&= \log \psi(k) + \phi(k) \lambda_1^d \quad \text{(definition of $\lambda^d$ and~\eqref{eqn-gf-truncated})} .
\alle

Since $\lim_{k\to\infty} (\log \psi(k)) / \phi(k) = 0$, there exists $K\in\BB N$ (not depending on $d$) such that 
\eqbn
\log \psi(k) \leq (\delta/2) \phi(k) ,\quad \forall k \geq K .
\eqen
If $d\in \BB N$ such that $k(d) \leq K-1$, then~\eqref{eqn-fixed-pt-one-coord} implies that 
\eqbn
\max_{k\in \{1,\dots,K\}} \lambda_k^d \geq \lambda_{k(d)}^d \geq \delta \phi(k(d)) \geq \delta .
\eqen

On the other hand, if $k(d) \geq K$, we combine~\eqref{eqn-fixed-pt-one-coord} and~\eqref{eqn-fixed-pt-exp-bd} to get
\eqb  \label{eqn-fixed-pt-coord-compare}
\log \psi(k(d) ) + \phi(k(d) )  \lambda_1^d  \geq \lambda_{k(d)}^d \geq \delta \phi(k(d))  . 
\eqe 
 By our choice of $K$, we have $\log \psi(k(d) ) \leq (\delta/2) \phi(k(d) )$, so~\eqref{eqn-fixed-pt-coord-compare} gives  
\eqbn
\lambda_1^d \geq  \frac{\delta}{2}  .
\eqen 
\end{proof}

To get a solution to $g(\lambda) = \lambda$ from a subsequential limit of solutions to $g^d(\lambda^d) = \lambda^d$, we will need to interchange limits and expectations. For this purpose we need an upper bound for $\lambda^d$, which is provided by the next lemma. We also get a relation between $\lambda^d$ and the fixed points of the single-type truncations of $X$, which is of independent interest. 

\begin{lem} \label{lem-1d-soln}
Let $k\in \BB N$ such that $\BB P_k[X_k(1) \geq 2] > 0$, which in particular holds for $k=1$. There exists a unique $\eta_k  \in (0,\infty)$ satisfying 
\eqb  \label{eqn-1d-soln} 
g_k(0,\dots, 0,\eta_k , 0 ,0,\dots) = \eta_k 
\eqe 
where here the input of $g_k$ has 0s in all components but the $k$th, which is $\eta_k $. 
Furthermore, for each $d\geq k$, 
\eqb  \label{eqn-1d-soln-ineq}
\lambda_k^d \leq \eta_k  .
\eqe  
Finally, if $t > 1$ is as in Assumption~\ref{item-gf-perturb-infty} and $\phi,\psi$ are as in Assumption~\ref{item-gf-growth-infty}, then for all integers $d \geq k \geq 1 $, 
\eqb  \label{eqn-1d-soln-growth}
\lambda_k^d \leq \phi(k) \log t + \log\psi(k)  .
\eqe  
\end{lem}
\begin{proof}
We first note that $\BB P_1[X_1(1) \geq 2] >0$ since the truncated branching process $X^1$ satisfies the hypotheses of Theorem~\ref{thm-fixed-gf} (Assumption~\ref{item-gf-trunc-infty}), so in particular it is non-singular.

Now let $k\in \BB N$ such that $\BB P_k[X_k(1) \geq 2] > 0$. 
Consider the single type branching process $\wt X^k$ obtained from $X$ by immediately killing all individuals of type $j \in \BB N\setminus \{k\}$. I.e., an individual in $X$ corresponds to an individual in $\wt X^k$ if and only if the individual and all of its ancestors have type $k$. Similarly as in~\eqref{eqn-gf-truncated}, the generating function for $\wt X^k$ is
\eqbn
\wt f_k(x) = f_k(1,\dots,1,x,1,1,\dots) 
\eqen
where here the input of $f_k$ has 1s in all components but the $k$th, which is $x$. 
Since the assumptions of Theorem~\ref{thm-fixed-gf} are satisfied for the truncated branching process $X^k$ (Assumption~\ref{item-gf-trunc-infty}), it is immediate that $\wt X^k$ is subcritical and $\wt f_k$ is continuous from $(0,\infty]$ to $[0,\infty]$. Since we are also assuming that $\BB P_k[X_k(1) \geq 2] > 0$, the hypotheses of Proposition~\ref{prop-single-type} are satisfied for $\wt X^k$. That proposition then gives us the desired unique solution to~\eqref{eqn-1d-soln} (namely, $\eta_k$ is the log of the solution $q$ from Proposition~\ref{prop-single-type}). 

For each $k\geq d$, the definition~\eqref{eqn-fixed-pt-trunc} of $\lambda^d$ gives
\eqb \label{eqn-fixed-pt-first-coord}
e^{\lambda_k^d} 
= f_k(e^{\lambda_1^d} , \dots , e^{\lambda_d^d} , 1,1,\dots) 
\geq f_k(1,\dots,1 , e^{\lambda_k^d} , 1,1,\dots)
= \wt f_k(e^{\lambda_k^d}) .
\eqe 
Since $\wt f_k$ is convex (its second derivative is positive) and satisfies $\wt f_k(1) = 1$ and $\wt f_k(e^{\eta_k}) = e^{\eta_k}$, we have $x < \wt f_k(x)$ for all $x  >  e^{\eta_k} $. Hence~\eqref{eqn-fixed-pt-first-coord} implies that $\lambda_k^d \leq \eta_k$.

Finally, to prove~\eqref{eqn-1d-soln-growth}, we observe that by Assumption~\ref{item-gf-perturb-infty} we have $\wt f_1(t) > t$. By convexity, $\wt f_1(x) \leq x$ for all $x\in [1,e^{\eta_1}]$, so $\eta_1 < \log t$. Then for each $d \geq k \geq 1$, 
\alb 
\lambda_k^d 
&= g_k(\lambda_1^d,\dots,\lambda_d^d,0,0,\dots) \quad \text{(by~\eqref{eqn-fixed-pt-trunc})}  \notag\\
&\leq \phi(k) g_1(\lambda_1^d,\dots,\lambda_d^d,0,0,\dots) + \log \psi(k) \quad \text{(by Assumption~\ref{item-gf-growth-infty})}  \notag\\
&= \phi(k) \lambda_1^d + \log\psi(k)  \quad  \text{(by~\eqref{eqn-fixed-pt-trunc})} \notag\\ 
&\leq \phi(k) \eta_1 + \log\psi(k)  \quad  \text{(by~\eqref{eqn-1d-soln-ineq})} \notag\\
&\leq \phi(k) \log t  + \log\psi(k) \quad \text{(since $\eta_1 < \log t$)} .
\ale
\end{proof}

\begin{proof}[Proof of Theorem~\ref{thm-fixed-pt-infty}]
By Lemma~\ref{lem-fixed-limsup} and a standard compactness argument, there is a sequence $\mcl D$ of positive integers tending to $\infty$ such that
\eqb \label{eqn-ssl}
\lambda_k := \lim_{\mcl D \ni d\to\infty} \lambda_k^d \quad \text{exists} \quad \forall k \in \BB N .
\eqe 

Let $\lambda = \{\lambda_k\}_{k\in\BB N}$. 
By the definitions of $g^d$ and $\lambda^d$, for each $k\in \{1,\dots,d\}$, 
\eqb \label{eqn-before-fatou} 
e^{\lambda_k^d } = \BB E_k\left[ \exp\left( \sum_{j=1}^d \lambda_j^d X_j(1) \right) \right] .
\eqe 
As $\mcl D \ni d \to \infty$, we have $\lambda_j^d \to \lambda_j$. Since a.s.\ there are only finitely many indices $j$ for which $X_j(1)$ is non-zero (see~\eqref{eqn-finite-offspring}), 
\eqb  \label{eqn-fixed-pt-as-conv}
\lim_{\mcl D \ni d \to\infty} \exp\left( \sum_{j=1}^d \lambda_j^d X_j(1) \right) = \exp\left( \sum_{j=1}^\infty \lambda_j X_j(1) \right) ,\quad \text{a.s.} .
\eqe

To show that $g (\lambda) = \lambda$, we need to interchange expectations with the limit in~\eqref{eqn-fixed-pt-as-conv}.  
By Lemma~\ref{lem-1d-soln}, we have $\lambda_j^d \leq \phi(j) \log t  + \log\psi(j)$ for each $d\geq j $. Hence
\eqb  \label{eqn-dominate}
\exp\left( \sum_{j=1}^d \lambda_j^d X_j(1) \right)
\leq \prod_{j=1}^\infty [ \psi(j) t^{\phi(j)}]^{X_j(1)} .
\eqe 
By Assumption~\ref{item-gf-perturb-infty}, the $\BB E_k$-expectation of the right side~\eqref{eqn-dominate} is finite for all $k\geq 1$. By~\eqref{eqn-before-fatou} and~\eqref{eqn-fixed-pt-as-conv} and the dominated convergence theorem, 
\eqbn
\lambda_k = \lim_{\mcl D \ni d \to\infty}  \lambda_k^d  =  \lim_{\mcl D \ni d \to\infty}  g_k(\lambda^d)  = g_k(\lambda) \quad \forall k \in \BB N . 
\eqen
In other words, $g(\lambda) = \lambda$. 
 
By Lemma~\ref{lem-fixed-liminf}, there exists $k \in \{1,\dots,K\}$ such that $\lambda_k > 0$. Hence, $\lambda \in [0,\infty)^{\BB N}\setminus \{0\}$. Since each truncated branching process $X^d$ is positive regular (Definition~\ref{def-bp-basic}), an argument similar to that in the second paragraph of the proof of Lemma~\ref{lem-fixed-pos} shows that in fact $\lambda \in (0,\infty)^{\BB N}$. 
\end{proof}

\begin{example} \label{example-ctble}
We describe a simple class of examples where Theorem~\ref{thm-fixed-pt-infty} implies the existence of a conjugate branching process. These examples can be viewed as toy models for the branching processes associated with the random planar map models discussed in Section~\ref{sec-lqg}.  
We specify a countable type branching process by specifying the $\BB P_k$-law of $\{X_j(1)\}_{j\in\BB N}$ for each $k\in \BB N$.
Let $\BB P_1$ be a probability measure on sequences of non-negative integers $\{X_j(1)\}_{j\in\BB N}$ such that 
\allb  \label{eqn-example-condition}
\BB P_1[X_j(1) \geq 1 ] > 0 ,\quad \forall j \in \BB N ,  \quad 
\BB P_1\left[  X_1(1) \geq 2 \right] > 0 , \notag\\
\BB E_1\left[ u^{\sum_{j=1}^\infty j X_j(1)} \right] < \infty ,\quad\forall u > 1 , \quad \text{and} \quad 
\BB E_1\left[ \sum_{j=1}^\infty j X_j(1) \right] < 1 .
\alle 
Note that these conditions imply that $\BB P_1$-a.s.\ only finitely many of the random variables $X_j(1)$ are non-zero. 
As a concrete example, we can first sample a random variable $N$ from a distribution supported on $\BB N_0$ which has mean less than 1 and finite exponential moments of all orders. Then, conditional on $N$, we sample $\{X_j(1)\}_{j\geq 1}$ uniformly from the set of all sequences of positive integers satisfying $\sum_{j=1}^\infty j X_j(1) = N$. 
For each $k\geq 2$, let $\BB P_k$ be a probability measure on sequences $\{X_j(1)\}_{j\in\BB N}$ which satisfies~\eqref{eqn-example-condition} (with $\BB E_k$ in place of $\BB E_1$) and which is stochastically dominated by the (componentwise) sum of $k$ i.i.d.\ samples from $\BB P_1$.  
One can easily verify that for the branching process with offspring distributions $\{\BB P_k\}_{k\in\BB N}$, the assumptions of Theorem~\ref{thm-fixed-pt-infty} are satisfied with $\phi(k) =k$, $\psi(k) = 1$, and $s,t > 1$ chosen so that
\eqbn
\BB E_1\left[ s^{ \sum_{j=1}^\infty j X_j(1) } \right] < s \quad \text{and} \quad
\BB E_1\left[ t^{  X_1(1) } \right]  > t  
\eqen
(existence of such an $s$ and $t$ follows easily from~\eqref{eqn-example-condition}).
\end{example}

\section{Applications}
 
\subsection{Applications to Liouville quantum gravity}
\label{sec-lqg}

One motivation for studying conjugates of subcritical branching process comes from the theory of random planar maps and Liouville quantum gravity (LQG). In this section we explain this motivation. We will not provide a detailed introduction to random planar maps and LQG in this paper. Instead, we refer the reader to~\cite{bp-lqg-notes,gwynne-ams-survey,sheffield-icm} for introductory expository articles on these topics.

\subsubsection*{Loop-decorated random planar maps and their gasket decomposition}

Let $k\in\BB N$. Consider a random planar map $M_k$ with boundary of perimeter $k$ (not necessarily simple), i.e., there are $k$ edges on the boundary of the external face. Let $\mcl L_k$ be a collection of vertex-disjoint loops (simple cycles) on the dual map $M_k^*$. We require that there is a special loop $\ell_\bdy \in \mcl L_k$ which visits every face of $M_k$ incident to $\bdy M_k$. See Figure~\ref{fig-loop-decorated}, left, for an illustration. We note that similar statements to the ones in this subsection also hold in slightly different setups.

\begin{figure}[ht!]
\begin{center}
\includegraphics[width=0.75\textwidth]{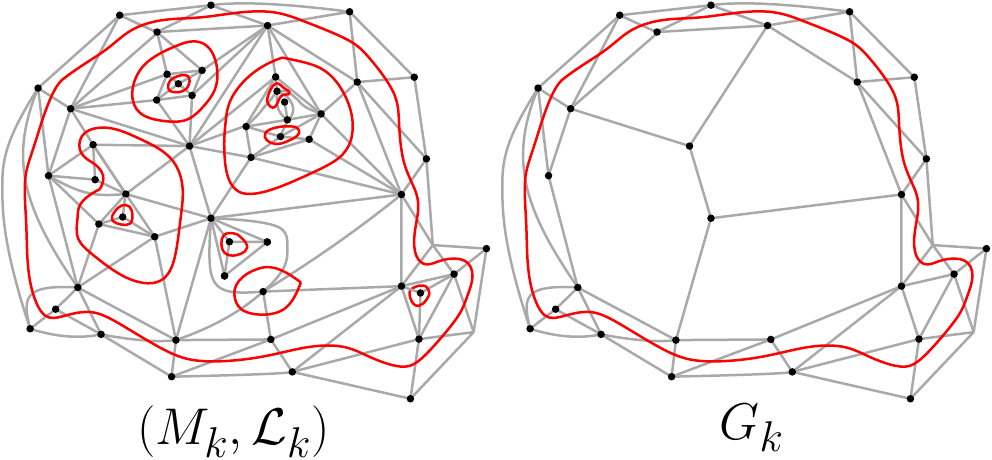}  
\caption{\label{fig-loop-decorated} A loop-decorated random planar map $(M_k,\mcl L_k)$ and its gasket $G_k$. 
}
\end{center}
\end{figure}

For a number of natural loop-decorated random planar maps $(M_k,\mcl L_k)$ of this type, including random planar maps decorated by versions of the $O(n)$ loop model\footnote{One possible variant of random planar maps decorated by the $O(n)$ loop model is as follows. Consider the set $\mcl T_k$ of all triangulations $M_k$ with boundary of perimeter $k$, decorated by a collection of vertex-disjoint loops $\mcl L_k$ on the dual map as above. Let $x,y > 0$ be parameters. Sample a pair $(M_k , \mcl L_k) \in \mcl T_k$ with probability proportional to $x^{\# \mcl A(M_k)} y^{\# \mcl B(M_k)} n^{\#\mcl L_k}$, where $\mcl A(M_k)$ (resp.\ $\mcl B(M_k)$) is the set of non-external faces (triangles) of $M_k$ which are (resp.\ are not) visited by loops. We need to assume that $x$ and $y$ are sufficiently small so that this defines a finite measure. To get convergence to LQG decorated by CLE, we need to further assume that the parameters $x$ and $y$ are tuned to certain critical values. See~\cite{shef-cle,bbg-recursive-approach,bbg-loop-models,bbg-bending} for further discussion on this point.}
with parameter $n > 0$, the pair $(M_k,\mcl L_k)$ can be described by a multi-type branching process via the so-called \textbf{gasket decomposition}~\cite{bbg-recursive-approach,bbg-loop-models,bbg-bending}. 

To explain how this works, we define the \textbf{gasket} $G_k$ of $(M_k,\mcl L_k)$, i.e., the planar map obtained from $M_k$ by deleting each face which is visited by a loop in $\mcl L_k \setminus \{\ell_\bdy\}$, as well as each vertex, edge, or face which is disconnected from $\ell_\bdy$ by such a loop. See Figure~\ref{fig-loop-decorated}, right.  
Then, for certain random pairs $(M_k,\mcl L_k)$, the loop-decorated submaps of $M_k$ contained in the ``holes'' of $G_k$ (i.e., the regions where we removed faces) are conditionally independent given $G_k$. The conditional law of each such loop-decorated submap enclosed by an outermost loop of perimeter $p$ is the same as the law of $(M_p , \mcl L_p)$ (i.e., it is a loop-decorated random planar map sampled in the same way as the original map, but with perimeter $p$ instead of $k$). Finally, the conditional law of $G_k$ given the lengths of the holes in $\mcl L_k$ is uniform over all possibilities. 

By applying this iteratively, we get that the lengths of the loops in $\mcl L_k$ are described by a multi-type branching process. The type of each loop is its length. The starting loop (at time zero) is $\ell_\bdy$. The offspring of each loop are the outermost loops in the submap it surrounds. 
For most loop-decorated random planar maps considered in the literature, the planar map $M_k$ is a.s.\ finite, which implies that the extinction probability of this branching process is one. 

In many interesting cases (including, e.g., random planar maps decorated by variants of the $O(n)$ loop model for $n\in (0,2]$), if the parameters for the model are tuned to their critical values, then $(M_k,\mcl L_k)$ is believed to converge under an appropriate scaling limit to a \textbf{$\gamma$-Liouville quantum gravity (LQG) surface} with the disk topology, with parameter $\gamma \in (\sqrt{2},2]$ decorated by a \textbf{conformal loop ensemble (CLE$_\kappa$)} with parameter $\kappa \in \{\gamma^2,16/\gamma^2\}$. One way to formulate this convergence is to first embed the random planar map into the unit disk in some way, e.g., via the circle packing~\cite{nachmias-circle-packing} or the Tutte embedding~\cite{gms-tutte}. The conjecture is that under this embedding, the counting measure on vertices, re-scaled by the total number of vertices, converges to the LQG area measure, and the embedded loops in $\mcl L_k$ converge to CLE$_\kappa$. This type of convergence has been proven in a few special cases~\cite{gms-tutte,hs-cardy-embedding}, but is still open for most models. See, e.g.,~\cite[Section 4]{bp-lqg-notes} for further discussion.

See~\cite{shef-cle} for the definition of CLE and~\cite{msw-non-simple-cle,msw-simple-cle-lqg} for the continuum analog of the gasket decomposition in the setting of CLE-decorated LQG.

\subsubsection*{Supercritical LQG}

Recent works have initiated the rigorous study of LQG in the \textbf{supercritical} phase, which corresponds to parameter values $\gamma\in\BB C$ with $|\gamma|=2$, or equivalently values of the central charge parameter between 1 and 25~\cite{ghpr-central-charge,dg-supercritical-lfpp,pfeffer-supercritical-lqg,dg-uniqueness,ag-supercritical-cle4,bgs-supercritical-crt}. This phase is much more mysterious than the case when $\gamma\in (0,2]$, even at a physics level of rigor, but is potentially more interesting than the case when $\gamma\in (0,2]$ from the perspective of string theory and Yang-Mills theory (see~\cite[Remark 1.4]{ag-supercritical-cle4} for further discussion of this point). Unlike LQG surfaces with $\gamma \in (0,2]$, supercritical LQG surfaces are not homeomorphic to two-dimensional manifolds. Rather, supercritical LQG surfaces have infinitely many infinite spikes. More precisely, if we parametrize a supercritical LQG surface by a domain $U\subset\BB C$, then there is an uncountable, Euclidean-dense, zero Lebesgue measure set of ``singular points'' in $U$ which lie at infinite distance from every other point with respect to the supercritical LQG metric.  

Particularly relevant to the present paper is~\cite{ag-supercritical-cle4}, which gives a continuum analog of the gasket decomposition for supercritical LQG surfaces decorated by CLE$_4$. Unlike in the case when $\gamma \in (0,2]$, the continuous state branching process involved in this model has supercritical behavior, e.g., in the sense that for every $r > 0$, the number of loops of supercritical LQG length at least $r$ in the $n$th generation a.s.\ goes to $\infty$ as $n\to\infty$~\cite[Proposition 3.1]{bgs-supercritical-crt} (this is related to the presence of infinite spikes in supercritical LQG surfaces). 

It is natural to ask what sorts of random planar maps converge to supercritical LQG surfaces, possibly decorated by CLE$_4$. A natural first thing to try is to look at one-parameter families of random planar map models which converge to $\gamma$-LQG for $\gamma\in (0,2]$ and extend the range of allowable parameter values.
For example, one can look at random planar maps sampled with probability proportional to $(\det \Delta)^{-(26-\cc )/2}$, where $\Delta$ is the discrete Laplacian and $\cc\in (1,25)$ is the central charge.

However, heuristic arguments and numerical simulations suggest that many interesting random planar maps obtained in this way behave like \textit{branched polymers}, meaning that they converge in the scaling limit to Aldous' continuum random tree~\cite{aldous-crt1} (which of course has very different geometry from supercritical LQG). 
See~\cite[Section 2.2]{ghpr-central-charge} and~\cite[Section 1.1]{bgs-supercritical-crt} for further discussion and references related to the physics literature on this topic and~\cite{feng-fk-triviality, bgs-supercritical-crt} for some rigorous results.

\subsubsection*{Constructing random planar maps via conjugate branching processes}

A potential explanation for why various natural random planar maps converge to the continuum random tree, instead of to supercritical LQG, was proposed in~\cite[Section 2.2]{ghpr-central-charge} (using some ideas from~\cite{apps-central-charge}) and made rigorous for a class of random planar maps in~\cite{bgs-supercritical-crt}. Loosely speaking, the reason is that the random planar maps considered in the physics literature always have a finite total number of vertices and edges. On the other hand, random planar maps whose scaling limit is supercritical LQG should be infinite, with infinitely many ends\footnote{
An \emph{end} of a graph $G$ is a direction in which the graph extends to infinity. Formally, a ray is an equivalence class of simple paths $P : \BB N \to \mcl V(G)$, where two simple paths $P$ and $\wt P$ are declared to be equivalent if and only if there is a finite set of vertices $X$ of $G$ and two distinct connected components $A$ and $B$ of $G\setminus X$ such that $A$ contains infinitely many vertices of $P$ and $B$ contains infinitely many vertices of $\wt P$.
}
(at least with high probability).
In particular, such random planar maps cannot be defined directly via a Radon-Nikodym derivative with respect to the uniform measure on some countable set of possibilities (which is the most common way to define random planar maps in the literature). 

The first combinatorially natural family of random planar map models which should converge to supercritical LQG was constructed in~\cite[Section 3.2]{ag-supercritical-cle4}. These random planar maps are constructed via a version of the gasket decomposition, but using a supercritical multi-type branching process rather than a critical or subcritical multi-type branching process. 
 
On the other hand, if we start with a random planar map which converges to supercritical LQG and condition on the (small-probability) event that it is finite, then its scaling limit should be the continuum random tree. This explains why various random planar maps which at first glance appear to be related to supercritical LQG instead converge to the continuum random tree. This was made rigorous for a class of random planar maps in~\cite{bgs-supercritical-crt}. More precisely,~\cite[Theorem 1.4]{bgs-supercritical-crt} shows that for a family of random planar maps believed to converge to supercritical LQG --- constructed from a supercritical multi-type branching process as in~\cite{ag-supercritical-cle4} --- one does indeed have convergence to the CRT when one conditions the map to be finite.

In summary, it is non-trivial to construct random planar maps whose scaling limit is supercritical LQG. However, it is easy to construct random planar maps which should be obtained from random planar maps which converge to supercritical LQG by conditioning on the event that the latter random planar map is finite. Hence, to construct additional models of random planar maps which converge to supercritical LQG, which are potentially more combinatorially natural than the ones in~\cite{ag-supercritical-cle4,bgs-supercritical-crt}, we want to find a way to ``invert the operation of conditioning a random planar map to be finite''. In light of the gasket decomposition for random planar maps, and its continuum analog for supercritical LQG and CLE$_4$ from~\cite{ag-supercritical-cle4}, a reasonable approach to achieving this is as follows.
\begin{itemize}
\item Start with a natural model of finite loop-decorated random planar maps with boundary which can be described in terms of a subcritical multi-type branching process via the gasket decomposition. We expect that the scaling limit of this random planar map should be the CRT.
\item Find a canonical supercritical multi-type branching process (or a canonical family of supercritical multi-type branching process) with the property that conditioning this supercritical branching process on extinction gives us back the original subcritical branching process. In other words, find a conjugate of the original branching process (in the sense of Definition~\ref{def-conjugate}). 
\item Construct a loop-decorated random planar map using the gasket decomposition, but with the conjugate supercritical branching process in place of the original subcritical branching process. This loop-decorated random planar map will be infinite, with infinitely many ends, with high probability if the initial perimeter is large. Moreover, conditioning on the rare event that this planar map is finite will give us back the original finite random planar map.  
\end{itemize} 
In order to carry out this procedure, one needs to understand when a conjugate for a subcritical branching process exists, which is the question we consider in the present paper. The other major step is a combinatorial analysis of the particular random planar map model in order to check the criteria for a conjugate branching process to exist. 

As a concrete example, we have reason to believe that this procedure should be applicable, and should lead to a random planar map whose scaling limit is supercritical LQG, for four-valent planar maps decorated by the six-vertex model in the case when the weight $C$ assigned to each alternating vertex satisfies $C > 2$. In the case when $C\in [0,2]$, the scaling limit of this model is believed to be critical ($\gamma=2$) LQG~\cite{zinn-justin-six-vertex,kostov-six-vertex}. 

The potential non-uniqueness of conjugate branching processes (recall Example~\ref{counterexample}) is interesting in the random planar map setting. It is reasonable to guess that for many natural models of finite loop-decorated random planar maps, the conjugate branching process is unique, even though this is not true for all multi-type branching processes. Another possible scenario is that for some models of finite loop-decorated random planar maps, the conjugate branching process is not unique, and we get multiple models of infinite random planar maps which all converge to supercritical LQG.

\subsection{Applications to biology} \label{sec-bio}
Applications of multi-type branching process in biology also motivates our study of the conjugate branching process.  
Multi-type branching processes provide a natural framework for studying how different types of organisms within a population reproduce, interact, and contribute to overall population growth. They have been particularly successful in modeling tumor growth. Below, we explain how the conjugate branching process can be useful in this context.

\subsubsection*{Cancer growth modeled by a multi-type branching process}
It is now well accepted that carcinogenesis is a multistage process. Cells accumulate mutations in genes that increase their fitness, or replication rate, and ultimately transforming them into tumor cells. This progression suggests that the clonal expansion of cancer cells can be modeled as a multi-type branching process, where cells of the same type share a specific set of mutations and therefore the same offspring distribution, and a cell produces cells of other types through additional mutations. This branching mechanism has been observed in many human cancers, including leukemia, breast, liver, colorectal, ovarian, prostate, kidney, melanoma, and brain cancers. See \cite[Section 6]{Davis2017} and the references therein. 

By changing types and offspring distributions, multi-type branching processes can model a range of cancer-related phenomena. In cancer initiation, types can represent stages in the cancer progression. For example, colon cancer progresses through four steps. The first two steps are the sequential inactivation of two tumor-suppressor genes. The third step is oncogene activation. The final step is the inactivation of a ``housekeeping gene'' that regulates cell division and DNA repair. This progression can thus be modeled as a four-type branching process. Moreover, drug resistance can be modeled as a two-type branching process, with one type representing drug sensitive cells and the other type representing drug resistant cells. Metastasis, the process by which cancer cells leave the original tumor, spread to other parts of the body, and forms new tumors, can be modeled as a three-type branching process. The three types are ordinary tumor cells, cells capable of metastasis but haven't done so, and cells that have spread to other locations. For detailed construction and results, see \cite{Durrett2015cancer} and the references therein.

The above models can be summarized as a multi-type branching process as follows. Consider a $d$-type branching process $X=\{X(n)\}_{n\in \BB N_0}$ where $X(n)=(X_1(n), \cdots, X_d(n))$. At each time step, a type $j$ cell produces an offspring of the same type with probability $b_j$, dies with probability $d_j$, or gives brith to an offspring of type $j+1$ that acquires an additional mutation with probability $u_j$. We note that $u_j$ is related to the mutation rate, which typically ranges from $10^{-9}$ to $10^{-5}$ and is significantly smaller than the other two probabilities. Most cancer models are in the continuous time setting. However, for simulation purposes, discrete time model like the one considered here are also used, for example in \cite{Bozic2010}.

\subsubsection*{Conjugate branching process provides alternative evolutionary hypothesis}
Cancer is a genetic disease. To understand the tumor growth in a specific patient, we need to understand both the genotype and phenotype of cancer cells within the tumor. This is challenging because tumor cells exhibit high heterogeneity and genotype-phenotype relationship in cancer is still largely unknown.
There are two categories of mutations during the natural progression of cancer, driver mutations which provide a reproductive advantage to cancer cells, and passenger mutations which do not contribute to cancer development directly and instead accumulate randomly as the cancer cell replicates. 
In addition to driver mutations which initiate cancer, patients could acquire lots of passenger mutations which are neutral and don't alter the reproduction rate. 
When the environment changes, as it does after cancer treatment, some passenger mutation may exhibit resistance that allows cells to survive from the therapy. 
Finding passenger mutations is much easier than driver mutations due to the vast genome, yet their high diversity complicates the assessment of their roles in cancer evolution. 
To addressing these complexities, a feasible approach is to first formally define the underlying assumptions of the tumor’s evolutionary dynamics and then statistically evaluate these hypotheses as suggested in \cite[Section 4]{Sottoriva2017}. 
Bayesian inference is particularly suited for this purpose. 
From the evolutionary hypothesis, we have prior knowledge about model parameters. Using Bayesian inference, we can test how well the model explains the observed data.
The conjugate branching process can be useful in providing alternative evolutionary hypothesis so that the model can be adjusted to better fit complex multi-dimensional data. 
As an example, we illustrate how the conjugate branching process can help detect genotypes resistant to cancer treatment at an early stage.
 
Currently, there is no universal cancer treatment effective for all patients with the same cancer type. Treatment failure often results from drug resistance, which can either exist before treatment (intrinsic resistance) or develop after treatment begins (acquired resistance). 
For example, targeted therapy employs drugs or other agents to target specific changes in cancer cells that promote growth, division, and spread. Despite often showing significant efficacy for most patients at the start, some patients may exhibit insensitivity from the beginning due to pre-existing genetic alterations that are resistant to the treatment. 
Moreover, some patients who respond well initially may later become resistant because the majority of the tumor cells are sensitive to the treatment, but the resistant cells would proliferate over time and result in recurrence. 
Early detection of potentially resistant genotypes is therefore crucial, as it could inform timely adjustments in therapeutic strategies. 
However, to pinpoint specific drug-resistant genotypes, there are two main difficulties.  
First, most cancers require immediate treatment upon diagnosis. Therefore the evolutionary trajectory of a tumor before treatment is generally unclear, except in a few well-studied cases like colorectal and colon cancers. 
Second, there is extensive genomic diversity among cancer cells, and phenotype such as drug resistance, does not have a one to one correspondence with genotypes. 
The conjugate branching process offers a pathway to identifying genotypes present in current patients that may be resistant to treatment, regardless of whether they emerged before or during initial therapy.

Suppose we use a multi-type branching process described above to model cancer evolution after a specific treatment. Here, the number of types $d$ corresponds to the number of distinct phenotypes observed across patients. In cancer evolution, while it is feasible to measure genotypes, phenotypes cannot be measured directly, and genotypes often include redundancies due to the abundance of neutral mutations. We want to include all relevant types in cancer evolution without overcomplicating the model. Thus, it is in practice a challenging task to select an appropriate number of types in the multi-type branching model. To analyze the model, we estimate from data the offspring distribution, or equivalently parameters $(b_j)_{j=1}^d$, $(d_j)_{j=1}^d$ and $(u_j)_{j=1}^d$. If this treatment proves effective for a substantial fraction of patients, then the empirical offspring distribution may suggest a subcritical multi-type branching process that is destined for extinction.

However, this is not true because an effective treatment does not guarantee complete eradication of the cancer or prevent relapse in all cases. 
If a treatment is effective for a patient, it suggests that the majority of their tumor cells do not have drug resistance mutations and the number of cells decreases below a detectable threshold. Thus, observing data from a successfully treated patient can be interpreted as conditioning on extinction.
This implies that the subcritical empirical offspring distribution we observe is the result of a supercritical branching process conditioned on extinction. By Theorem \ref{thm-fixed-gf}, we know such supercritical processes exist. 
Moreover, since the conjugate supercritical branching process is not unique, multiple evolutionary hypotheses may be considered and need to be tested against data.

By setting the conjugate supercritical branching process as the evolutionary hypothesis, we can further investigate the presence of intrinsic drug resistance and the likelihood of cancer recurrence.
Cancer cells may carry pre-existing passenger mutations that reduce their sensitivity to certain drugs. It is possible to infer genotypes that are resistant to drug from the conjugate supercritical branching process. 
Cancer often becomes undetectable when cell numbers fall below a certain threshold. Recurrence can occur if residual cancer cells survive initial treatment and proliferate. With the conjugate supercritical branching process, we can estimate the probability of cancer recurrence.
It can be approximated by the probability that the conjugate supercritical branching process starting from an initial population of drug insensitive cells of size $x$ proliferates and eventually surpasses the detection threshold $M$. Here $x$ is a lot smaller than $M$ and $M$ is typically of the order from $10^8$ to $10^9$ \cite{Del2009}.

Multi-type branching processes have been widely applied in various fields, including the spread of infectious diseases \cite{Jacob2010, Penisson2012}, gene amplification \cite{Harnevo1991}  and so on. Usually in the biology context, the number of types is finite. The conjugate branching process can be useful in these contexts. For any real-life phenomenon modeled by a multi-type branching process, empirical data may sometimes suggest a subcritical process where the offspring distribution inherently leads to population decline and eventual extinction. However, it is also possible that the data we observed is the result of a supercritical process conditioned on extinction. 
In cases where the latter scenario is suspected, Theorem \ref{thm-fixed-gf} provides mathematical justification for reconstructing the original supercritical processes, allowing us to estimate the actual offspring distribution and gain clearer insights into the underlying dynamics of the model.

\bibliography{cibib,extrabib}
\bibliographystyle{hmralphaabbrv}

\end{document}